\documentclass{article}
\usepackage[utf8]{inputenc}

\usepackage{amsfonts}
\usepackage{amssymb}
\usepackage{amsmath}
\usepackage{amsthm}
\usepackage{mathtools}
\usepackage{stmaryrd}
\usepackage{cleveref}
\usepackage{url}
\usepackage{float}
\usepackage{adjustbox}
\usepackage{booktabs}
\usepackage[nottoc]{tocbibind}
\usepackage{subcaption}

\usepackage{tikz}
\usepackage{tikz-cd}

\theoremstyle{plain}
    \newtheorem{maintheorem}{Theorem}
    \newtheorem{mainproposition}{Proposition}
    \newtheorem{mainquestion}{Question}
        
    \newtheorem{proposition}{Proposition}[section]
    \newtheorem{theorem}[proposition]{Theorem}
    \newtheorem{corollary}[proposition]{Corollary}
    \newtheorem{lemma}[proposition]{Lemma}

\theoremstyle{definition}
    \newtheorem{definition}[proposition]{Definition}
    \newtheorem{condition}[proposition]{Condition}

    \newtheorem{notation}[proposition]{Notation}
    \newtheorem{question}[proposition]{Question}

    \newtheorem{algorithm}[proposition]{Algorithm}
    \newtheorem{example}[proposition]{Example}

    \newtheorem{case}{Case}
    \newtheorem{subcase}{Case}[case]
	\newtheorem{claim}{Claim}

\theoremstyle{remark}
	\newtheorem{remark}[proposition]{Remark}%

\newcommand{\bulsubsection}[1]{\subsection*{$\bullet$\quad #1}}
\newcommand{\bulsubsubsection}[1]{\subsubsection*{$\bullet$\quad #1}}

\newcommand{\ZZ}{\mathbb{Z}}
\newcommand{\QQ}{\mathbb{Q}}
\newcommand{\RR}{\mathbb{R}}

\newcommand{\FF}{\mathbb{F}}

\newcommand{\rn}[1]{\romannumeral #1}

\newcommand{\id}{\mathit{id}}
\newcommand{\isom}{\cong}
\newcommand{\htpy}{\simeq}

\newcommand{\<}{\left<}
\renewcommand{\>}{\right>}

\newcommand{\ox}{\otimes}

\renewcommand{\epsilon}{\varepsilon}

\DeclareMathOperator{\Span}{Span}

\DeclareMathOperator{\fchar}{char}

\newcommand{\slfrac}[2]{\left.#1\middle/#2\right.}
    
\newcommand{\Kh}{\textit{Kh}}
\newcommand{\Lee}{\textit{Lee}}
\newcommand{\BN}{\textit{BN}}

\DeclareMathOperator{\qdeg}{qdeg}

\renewcommand{\bar}[1]{\mkern 1.5mu\overline{\mkern-1.5mu#1\mkern-1.5mu}\mkern 1.5mu}

\renewcommand{\a}{\mathbf{a}}
\renewcommand{\b}{\mathbf{b}}
\newcommand{\s}{\bar{s}}

\newcommand{\ca}{\alpha}
\newcommand{\cb}{\beta}
\newcommand{\cx}{\xi}
\newcommand{\cy}{\eta}
\newcommand{\cz}{\zeta}

\newcommand{\dottedcircle}{\tikz\draw[black,dotted] (0,0) circle (1ex); $\,$}

\title{
    Divisibility of Lee's class and its relation with Rasmussen's invariant
}
\author{Taketo Sano}

\begin{document}

    \maketitle
    
    \begin{abstract}
    Lee homology (a variant of Khovanov homology) over $\QQ$ possesses the ``canonical generators" as its basis. The generators (Lee's classes) $[\ca(D, o)]$ are constructed combinatorially from an oriented link diagram $D$, one for each alternative orientation $o$ on $D$. Let $R$ be an integral domain. There exists a family of link homology theory $\{ H_c(-; R) \}_{c \in R}$, where Khovanov's theory corresponds to $c = 0$ and Lee's theory corresponds to $c = 2$. For each $c \in R \setminus 0$, Lee's classes $[\ca(D, o)]$ can be defined as elements in $H_c(D; R)$, but when $c$ is not invertible then they do not form a basis; in fact they are divisible by $c$-powers. We define the $c$-divisibility $k_c(D)$ of $[\ca(D, o)]$ with $o$ the given orientation of $D$. For any link $L$ and its diagram $D$, we prove that $\bar{s}_c(L) := 2k_c(D) + w(D) - r(D) + 1$ is a link invariant, where $w$ is the writhe, and $r$ is the number of Seifert circles. We pose the question whether $\bar{s}_c$ coincides with Rasmussen's $s$-invariant. There are several evidences that support the affirmative answer. For instance, $\s_c$ is a link concordance invariant, and the Milnor conjecture can be reproved using $\s_c$. Also for the special case $(R, c) = (\QQ[h], h)$, our $\bar{s}_c$ actually coincides with $s$ as knot invariants.
\end{abstract}

    \tableofcontents
    
    \section{Introduction}

Almost two decades have passed since Khovanov introduced in \cite{khovanov2000} a link homology theory, now known as Khovanov homology, that categorifies the Jones polynomial. In \cite{rasmussen2010khovanov} Rasmussen introduced a knot invariant $s$ based on Lee homology (a variant of Khovanov homology introduced by Lee in \cite{lee2005endomorphism}). He proved: (\rn{1}) the invariant $s$ defines a homomorphism from the knot concordance group in $S^3$ to $2\ZZ$, (\rn{2}) it provides a lower bound for the slice genus $g_*$ of knots, and (\rn{3}) $s$ and $g_*$ are equal for positive knots. Then the Milnor conjecture \cite{Milnor1968} follows as a corollary. The conjecture was originally proved by Kronheimer and Mrowka in \cite{Kronheimer:1993} using gauge theory, but Rasmussen's result was notable since it provided for the first time a purely combinatorial proof.

\medskip

The well-definedness of $s$ is based on the invariance of the ``canonical generators" of Lee homology over $\QQ$. Let $L$ be an oriented link and $D$ be its diagram. The generators $[\ca(D, o)] \in H_\Lee(D; \QQ)$ are constructed combinatorially from $D$, one for each alternative orientation $o$ on $D$. Lee introduced these classes in \cite{lee2005endomorphism} and proved that they form a basis of $H_\Lee(D; \QQ)$. Rasmussen then proved that they are invariant (up to unit) under the Reidemeister moves, hence called them the canonical generators of $H_\Lee(L; \QQ)$. The construction can be done over $\ZZ$, so one may expect that they form a basis of $H_\Lee(D; \ZZ)$. However, by direct computation, we see that there are many diagrams such that the classes are divisible by 2-powers (hence do not form a generating set), and the 2-divisibility is not constant among diagrams of the same link. Where does `2' come from? How does 2-divisibility of the classes vary under the Reidemeister moves?

\medskip

In this paper, we show that `2' is the difference of the two roots of the polynomial $X^2 - 1$ which is used in the construction of Lee homology. This phenomenon can be seen in a more generalized setting. Khovanov introduced in \cite{khovanov2004} a family of link homology theories $\{ H_{h, t}(-; R) \}_{h, t \in R}$ over an arbitrary commutative ring $R$. Khovanov's original theory is given by $(h, t) = (0, 0)$ and Lee's theory is given by $(h, t) = (0, 1)$. From the arguments given by Mackaay, Turner, Vaz in \cite{mackaay2007remark}, we see that if the polynomial $X^2 - hX - t$ factors into linear polynomials, then Lee's classes $[\ca(D, o)]$ can be defined in $H_{h, t}(D; R)$.

\begin{mainproposition}
    With the above setting, let $c = \sqrt{h^2 + 4t}$.
    \begin{enumerate}
        \item If $c$ is invertible in $R$, then Lee's classes form a basis of $H_{h, t}(D; R)$. (\Cref{prop:ab-gen})
        
        \item If $D, D'$ are two diagrams related by a single Reidemeister move, then the `ratio' of the corresponding pair of classes $[\ca(D, o)]$ and $[\ca(D', o')]$ is given by $\pm c^j$ for some $j \in \{0, \pm 1\}$. (\Cref{prop:cacb-variance-under-rho})
        
        \item For another $(h', t')$ such that $c = \sqrt{h'^2 + 4t'}$, the corresponding groups $H_{h, t}(D; R)$ and $H_{h', t'}(D; R)$ are naturally isomorphic, and under the natural isomorphism Lee's classes correspond one-to-one. (\Cref{prop:ht-relation})
    \end{enumerate}
\end{mainproposition}

The first two statements imply that the situation is completely analogous to $\QQ$-Lee theory when $c$ is invertible. The third says that the isomorphism class of $H_{h, t}(D; R)$ is determined by $c$, and Lee's classes are well-defined under the identification. Thus we obtain a family of link homology groups $\{H_c(D; R)\}_{c \in R}$. \Cref{fig:ht-param-sp_intro} depicts the $(h, t)$-parameter space, where each point $(h, t)$ corresponds to $H_{h, t}(D; R)$ and the parabola $h^2 + 4t = c^2$ corresponds to the isomorphism class $H_c(D; R)$. 

\begin{figure}[ht]
    \centering
    \includegraphics[scale=0.35]{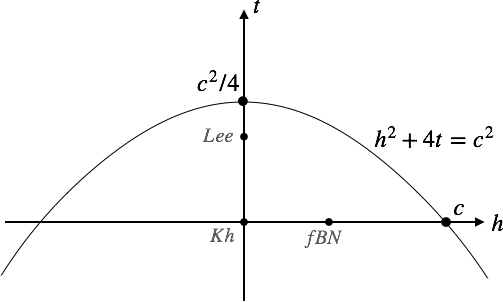}
    \caption{The $(h, t)$-parameter space.}
    \label{fig:ht-param-sp_intro}
\end{figure}

Next we assume $R$ is an integral domain, and consider when $c$ is non-zero, non-invertible in $R$. In this case, Lee's classes do not form a basis of $H_c(D; R)$; in fact they are divisible by $c$-powers. Let $o$ be the given orientation of $D$. We define the \textit{$c$-divisibility} of $[\ca(D, o)]$ (modulo torsions) by the exponent of its $c$-power factor, and denote it by $k_c(D)$. From the following proposition, we may regard $k_c$ as measuring the ``non-positivity" of the diagram. 

\begin{mainproposition}
    If $D$ is positive, then $k_c(D) = 0$. (\Cref{prop:k_posD})
\end{mainproposition}

In another paper \cite{Rasmussen:2005vo} of Rasmussen's, the variance of the canonical generators under the Reidemeister moves is stated in more detail. The statement is generalized to our case as follows:

\begin{mainproposition}
    The variance of $k_c$ under each Reidemeister move is equal to the variance of $\frac{r - w}{2}$, where $w$ is the writhe, and $r$ is number of Seifert circles. (\Cref{prop:cacb-variance-under-rho})
\end{mainproposition}

Thus we obtain:

\begin{maintheorem}
    Let $L$ be a link. For any diagram $D$ of $L$, the value
    \[
        \s_c(L) := 2k_c(D) + w(D) - r(D) + 1
    \]
    is an invariant of $L$.
\end{maintheorem}

From the computational experiments\footnote{
    \url{https://git.io/fphro}
} 
done for $(R, c) = (\ZZ, 2)$ (i.e.\ $\ZZ$-Lee theory), we see that the values of $\s_2$ coincide with those of $s$ for prime knots of crossing number up to 11. Here we pose the main question:

\begin{mainquestion} \label{mainq1}
    Does $\s_c$ coincide with $s$ for any $(R, c)$?
\end{mainquestion}

We give several theoritical evidences that support the affirmative answer. First, by inspecting the behavior of $\s_c$ under cobordisms, we obtain properties of $\s_c$ that are common to $s$:

\begin{maintheorem} $\ $
    \begin{enumerate}
        \item $\s_c$ is a link concordance invariant,
        \item $|\s_c(K)| \leq 2g_*(K)$ for any knot $K$, and
        \item $\s_c(K) = 2g_*(K) = 2g(K)$ if $K$ is positive.
    \end{enumerate}
\end{maintheorem}

As a corollary, the Milnor conjecture can be reproved using $\s_c$. Next, we show that there exists $(R, c)$ such that $\s_c$ coincides with $s$ as knot invariants. The original definition of $s$ uses $\QQ$-Lee homology, but it can be generalized over an arbitrary field $F$ of $\fchar{F} \neq 2$. For the special case $(R, c)$ = $(F[h], h)$ where $h$ is an indeterminate of degree $-2$ (this corresponds to the Bar-Natan theory over $F$), we prove:

\begin{maintheorem} \label{mainthm3}
    For any knot $K$, 
    \[
        s(K; F) = \s_h(K; F[h])
    \]
    where $s(K; F)$ is the $s$-invariant of $K$ over $F$.
\end{maintheorem}

It is a famous open question whether there exists any $F$ such that $s(-; F)$ is distinct from $s = s(-; \QQ)$ (\cite[Question 6.1]{lipshitz2014refinement}). \Cref{mainthm3} implies that if \Cref{mainq1} is solved affirmatively, then all $s(-; F)$ are equal among fields $F$ of $\fchar{F} \neq 2$.

\medskip

Viewing the $s$-invariant from the perspective of divisibility has been suggested by Kronheimer and Mrowka in \cite{Kronheimer:2011by}, and by Collari in \cite{Collari:2017wr}, both based on the alternative definition of $s$ given by Khovanov in \cite{khovanov2004}. We expect that our approach would also lead to a better understanding of $s$.

\subsection*{Outline}

In Section 2, we first review the construction of Khovanov homology and Lee's classes in the generalized setting. Then we state the variance of the classes under the Reidemeister moves. In Section 3, we define $k_c(D)$ for a link diagram $D$ and the invariant $\s_c(L)$ for a link $L$. We also state the behavior of $\s_c$ under cobordisms, and obtain properties common to Rasmussen's $s$-invariant. Then we focus on knots and specialize to $(R, c) = (F[h], h)$ where $F$ is a field of $\fchar{F} \neq 2$, and prove that $\s_h$ coincides with $s$ over $F$. The final section gives further remarks and questions. 

\subsection*{Conventions}

In this paper, we assume knots and links are oriented. For a link $L$, we denote by $|L|$ the number of components of $L$, by $-L$ the link obtained from $L$ by reversing the orientation on each of its components, and by $\bar{L}$ the mirror image of $L$. We use the same notations for link diagrams.

\subsection*{Acknowledgements}

This paper is based on the master's thesis submitted to the Graduate School of Mathematical Sciences, the University of Tokyo. I am deeply grateful to my supervisor Mikio Furuta for his guidance and for everything he taught me in mathematics. I am also grateful to Kouki Sato for devoting a large amount of time for helpful discussions. I thank Yukio Kametani, Takayuki Kobayashi, Ryusuke Horiuchi for introducing me to the subject; Yoshihiro Matsumori for constructive feedbacks and correction of errors; members of \textit{swift-developers-japan} for helping me improve the computer program; and the members of my \textit{academist fan club} \footnote{\url{https://taketo1024.jp/supporters}} for the financial support. Finally, I thank my wife and my daughter for their support and encouragement throughout the years of my study and research.

\setcounter{maintheorem}{0}
    \section{Khovanov homology and Lee's classes} \label{sec:gen-kh-theory}

\subsection{Khovanov homology} \label{subsec:frob-n-khovanov}

In this section, we review Khovanov homology theory in the generalized form as given in \cite{khovanov2004}. Let $R$ be a commutative ring with unity. A \textit{Frobenius algebra} over $R$ is a quintuple $(A, m, \iota, \Delta, \epsilon)$ such that: (\rn{1}) $(A, m, \iota)$ is an associative $R$-algebra with multiplication $m$ and unit $\iota$, (\rn{2}) $(A, \Delta, \epsilon)$ is a coassociative $R$-coalgebra with comultiplication $\Delta$ and counit $\epsilon$, and (\rn{3}) the Frobenius relation holds: 
\[
    \Delta \circ m = (\id \otimes m) \circ (\Delta \otimes \id) = (m \otimes \id) \circ (\id \otimes \Delta).
\]

Let $h, t$ be two elements of $R$. Define $A_{h, t} = R[X]/(X^2 - hX - t)$ with the obvious $R$-algebra structure. Define the counit $\epsilon: A_{h, t} \rightarrow R$ by
\[
    \epsilon(1) = 0,\quad
    \epsilon(X) = 1.
\]
Then the comultiplication $\Delta$ is uniquely determined so that $(A_{h, t}, m, \iota, \Delta, \epsilon)$ becomes a Frobenius algebra. Explicitly, $\Delta$ is given by
\[
    \Delta(1) = X \otimes 1 + 1 \otimes X - h1 \otimes 1, \quad 
	\Delta(X) = X \otimes X + t 1.
\]

Given a link diagram $D$, we obtain a chain complex $C_{h, t}(D; R)$ and its homology $H_{h, t}(D; R)$ by following the construction of the original Khovanov homology, except that the Frobenius algebra $R[X]/(X^2)$ is replaced with $A_{h, t}$. The chain complex is also given a secondary grading, and under some condition $H_{h, t}(D; R)$ becomes bigraded or filtered. Here we remark that Khovanov's original theory (\cite{khovanov2000}) and Lee's theory (\cite{lee2005endomorphism}) are given by
\[
    H_\Kh = H_{0, 0},\quad 
    H_\Lee = H_{0, 1},
\]
and Bar-Natan's theory (\cite{BarNatan:2004if}) is given by
\[
    H_\BN = H_{h, 0}
\]
where $h$ is an indeterminate of degree $-2$. By collapsing $H_\BN$ with $h = 1$ we obtain the filtered version:
\[
    H_\mathit{fBN} = H_{1, 0}.
\]

The following theorem assures that any $H_{h, t}$ gives a link invariant:

\begin{theorem}[{\cite[Proposition 6]{khovanov2004}}] \label{thm:t-inv}
    Let $L$ be a link. For any diagram $D$ of $L$, the isomorphism class of $H_{h, t}(D; R)$ as a (graded / bigraded / filtered) $R$-module is an invariant of $L$.
\end{theorem}

In order to fix the terms and notations used in the later sections, we briefly review the construction of the chain complex $C_{h, t}(D; R)$. Let $D$ be a link diagram with $n$ crossings. Each crossing admits a 0-, 1-resolution as depicted in  \Cref{fig:1}. 
\begin{figure}[t]
	\centering
    \includegraphics[scale=0.4]{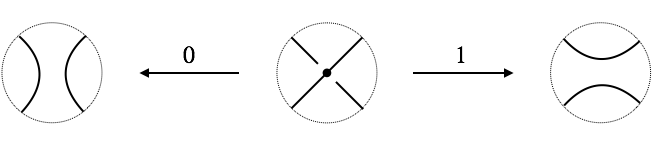}
	\caption{0-, 1-resolution of a crossing.}
	\label{fig:1}
\end{figure}
A simultaneous choice of resolutions for all crossings of $D$ is called a \textit{state}. For each state $s$, denote by $|s|$ the number of 1-resolutions in $s$. Two states are \textit{adjacent} if one is obtained from the other by changing the resolution of a single crossing. We write $s \prec s'$ when $s$ and $s'$ are adjacent and $|s| + 1 = |s'|$. Any state $s$ yields a diagram consisting of disjoint circles by resolving all crossings accordingly. We call such circles \textit{$s$-circles} and denote its number by $r(D, s)$. Each state $s$ corresponds to an $R$-module $V(D, s) = A^{\ox r(D, s)}$. An element in $V(D, s)$ of the form:
\[
    x = e_1 \ox \cdots \ox e_r,
    \quad e_i \in \{ 1, X \}
\]
is called an \textit{enhanced state}. For each pair of adjacent states $s \prec s'$, there is an $R$-module homomorphism $V(D, s) \rightarrow V(D, s')$, given by the multiplication $m$ or the comultiplication $\Delta$ of $A$, depending on whether the $s$-circle(s) merge or split when the resolution of the single crossing is changed from $0$ to $1$. These modules and maps form a commutative $n$-dimensional cubic diagram. After some adjustment of signs of the maps, the cube is folded to form the unnormalized chain complex
\[
    \bar{C}_{h, t}^i(D; R) = \bigoplus_{|s| = i} V(D, s)
\]
(see \Cref{fig:CKh}). Then $C_{h, t}(D)$ is defined by shifting the homological degree of $\bar{C}_{h, t}(D)$ by $-n^-$, the number of negative crossings of $D$. Each enhanced state $x$ is also endowed a secondary degree, which we call the \textit{q-degree}. Let $\deg(1) = 0,\ \deg(X) = -2$. Define $\deg(x) := \sum_i \deg(e_i)$ and
\[
    \qdeg(x) := \deg(x) + |s| + r(D, s) + n^+ - 2n^-,
\]
where $n^+, n^-$ are the number of positive, negative crossings of D respectively. This gives a bigrading on $C_{h, t}(D)$. If $R$ is graded with $h, t$ having 
\[
    (\deg{h} = -4 \text{\ or\ } h = 0) \text{\ and\ }
    (\deg{t} = -2 \text{\ or\ } t = 0),
\]
then the differential $d$ preserves the q-degree and $H_{h, t}(D; R)$ inherits the bigrading. Otherwise if 
\[
    (\deg{h} \geq -4 \text{\ or\ } h = 0) \text{\ and\ }
    (\deg{t} \geq -2 \text{\ or\ } t = 0),
\]
then $d$ is q-degree non-decreasing, so we may define a filtration on $C_{h, t}(D; R)$ as
\[
    F^j C_{h, t}^i(D; R) := \{ x \in C_{h, t}^i(D; R) \mid \qdeg(x) \geq j \}.
\]
$H_{h, t}(D; R)$ admits the induced filtration, where the (filtered) q-degree of a homology class $\gamma$ is given by the maximum q-degree among its representatives
\[
    \qdeg(\gamma) := \max\{ \qdeg(x) \mid [x] = \gamma \}. 
\]
Note that $H_\Kh$ and $H_\BN$ are bigraded, whereas $H_\Lee$ and $H_{f\BN}$ are filtered. 

\begin{figure}[t]
	\centering
    \includegraphics[scale=0.25]{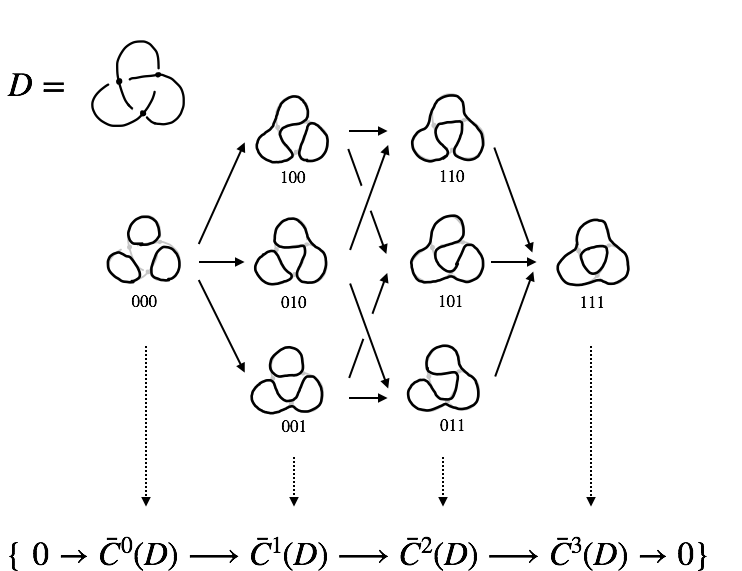}
	\caption{$\bar{C}^\cdot(D)$ for the left hand trefoil. \label{fig:CKh}}
\end{figure}

\medskip

Finally we state some basic properties of $C_{h, t}$ that are given in \cite{khovanov2000}.

\begin{proposition} \label{lem:ckh-supp}
	$C_{h, t}^{i, j}(D; R)$ is supported only where $j \equiv |D| \bmod{2}$. 
\end{proposition}

\begin{proposition} \label{prop:C_ht-basic}
    \ 
    \begin{enumerate}
        \item $C_{h, t}(-D) = C_{h, t}(D)$
        \item $C_{h, t}(D \sqcup D') \isom C_{h, t}(D) \otimes C_{h, t}(D')$
        \item $C_{h, t}(\bar{D}) \isom C_{-h, t}^*(D)$
    \end{enumerate}
    where $-D$ is $D$ with its orientation reversed, and $\bar{D}$ is the mirror image of $D$.
\end{proposition}

    \subsection{Lee's classes}

In \cite{lee2005endomorphism} Lee constructed a set of classes $[\ca(D, o)] \in H_\Lee(D; \QQ)$, one for each alternative orientation $o$ on $D$, and proved that they form a basis. By following the arguments given by Mackaay, Turner, Vaz in \cite{mackaay2007remark}, these classes can be generalized as elements in $H_{h, t}(D; R)$. Throughout this section, we assume the following condition holds: 

\begin{condition} \label{cond:ab-cond1}
	There exists $c \in R$ such that $h^2 + 4t = c^2$ and $(h \pm c)/2 \in R$.
\end{condition}
Fix one $c = \sqrt{h^2 + 4t}$, and let 
\[
    u = (h - c)/2,\quad 
    v = (h + c)/2 \ \in \ R.
\]
Then $X^2 - hX - t$ factors as $(X - u)(X - v)$ in $R[X]$. Also let 
\[
    \a = X - u,\quad 
    \b = X - v \ \in \ A.
\]
Then obviously $\a\b = \b\a = 0$. Also with $\a - \b = v - u = c$, we have:
\begin{alignat*}{2}
	m(\a \otimes \a) &= c\a, 
	    &\hspace{1cm} 
	     \Delta(\a) &= \a \otimes \a, \\
	m(\a \otimes \b) &= 0, 
	    &\Delta(\b) &= \b \otimes \b \\
	m(\b \otimes \a) &= 0 &&\\
	m(\b \otimes \b) &= -c\b &&
\end{alignat*}
Here we call $\a$ and $\b$ \textit{colors}. For any state $s$, a coloring on the $s$-circles defines an element in $V(D, s)$, which we call a \textit{colored state}. Recall that a link diagram possesses a unique \textit{orientation preserving state} $s$, where every state circle admits an orientation coherent with the given orientation of $D$. Such a state can be obtained by 0-resolving the positive crossings, and 1-resolving the negatives as in \Cref{fig:ori-pres-st}. The corresponding state circles are the \textit{Seifert circles}. We color the Seifert circles according to the following algorithm:

\begin{figure}[ht]
	\centering
    \includegraphics[scale=0.4]{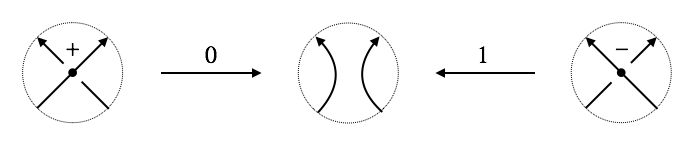}
	\caption{Orientation preserving resolution. \label{fig:ori-pres-st}}
\end{figure}

\begin{algorithm} \label{algo:ab-coloring}
    Color the regions of $\mathbb{R}^2$ divided by the Seifert circles in the checkerboard fashion, where the unbounded region is colored white. Color a  circle $\a$ if it sees a black region to the left with respect to the given orientation, otherwise color $\b$.
\end{algorithm}

\begin{figure}[ht]
	\centering
    \includegraphics[scale=0.35]{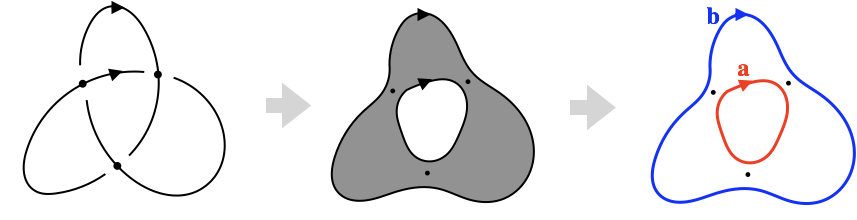}
	\caption{Coloring the Seifert circles by $\a$, $\b$. \label{fig:ab}}
\end{figure}

\begin{lemma} \label{lem:ab-bipartite}
    Every crossing of $D$ connects differently colored circles. In particular, no crossing connects a circle to itself. \qed
\end{lemma}


Denote by $\ca(D) \in V(D, s)$ the colored state obtained by \Cref{algo:ab-coloring}. If we forget the given orientation of $D$, there are $N = 2^{|D|}$ possible orientations on the underlying unoriented diagram of $D$. We call each of them an \textit{alternative orientation} of $D$. For each alternative orientation $o$, there is the corresponding orientation preserving state $s_o$, and we obtain an element $\ca(D, o) \in V(D, s_o)$ by the same procedure. 

\begin{proposition} \label{lem:canon-cycle-is-a-cycle}
    Each $\ca(D, o)$ is a cycle in $C_{h, t}(D; R)$. \qed
\end{proposition}


\begin{definition}[$\ca$-cycles, $\ca$-classes]
    We call the cycles $\ca(D, o)$ the \textit{$\ca$-cycles} of $D$, and the homology classes $[\ca(D, o)]$ the \textit{$\ca$-classes} of $D$. If $o$ is the given orientation of $D$, we simply denote the corresponding cycle by $\ca(D)$. We call it the \textit{$\ca$-cycle} of $D$, and $[\ca(D)]$ the \textit{$\ca$-class} of $D$. 
\end{definition}

Lee proved in \cite{lee2005endomorphism} that the $\QQ$-Lee homology of $D$ is freely generated by the $\ca$-classes, so in particular $H_\Lee(D; \QQ) \isom \QQ^N$. This generalizes as:

\begin{proposition} \label{prop:ab-gen}
	If $c = \sqrt{h^2 + 4t}$ is invertible in $R$, then $H_{h, t}(D; R)$ is freely generated over $R$ by the $\ca$-classes. In particular $H_{h, t}(D; R) \isom R^N$.
\end{proposition}

Lee's proof cannot be applied directly, since it uses Hodge theory and requires that $R$ is a field. However there is an alternative proof by the \textit{admissible coloring decomposition} of $C_{h, t}(D; R)$, proposed by Wehrli in \cite[Remark 5.4]{Wehrli:uy}. We will briefly explain how this is applicable to our case. A \textit{coloring} of a diagram $D$ is an assignment of either $\a$ or $\b$ on each arc of $D$. A coloring is \textit{admissible} if each crossing admits a resolution such that the arc segments can be colored accordingly. For an admissibly colored diagram, every crossing is locally colored as one of the three of \Cref{fig:loc-col}. Now if $c$ is invertible, then the colored states form a basis of $C_{h, t}(D; R)$. The idea of the proof is to decompose $C_{h, t}(D; R)$ by admissible colorings, and prove that the homology is exactly the subcomplex generated by the $\ca$-classes (the remaining part is acyclic). A detailed proof can be found in Lewark's paper \cite[Lemma I.14]{lewark2009rasmussen}.

\begin{figure}[ht]
	\centering
    \includegraphics[scale=0.5]{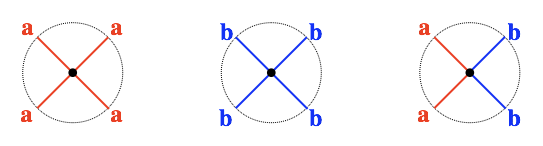}
	\caption{Local colorings of an admissibly colored diagram. \label{fig:loc-col}}
\end{figure}

\begin{corollary} \label{cor:H(D)-tor}
	If $c = \sqrt{h^2 + 4t} \neq 0$, then $H_{h, t}(D; R)$ contains only $c$-torsions, \text{i.e.} all torsions are annihilated by multiplying some power of $c$.
\end{corollary}

\begin{proof}
    Let $R_c = R[c^{-1}]$ be the ring of $R$ localized by powers of $c$. Since $R_c$ is flat over $R$ we have $H_{h, t}(D; R) \otimes R_c \isom H_{h, t}(D; R_c) \isom (R_c)^N$. The result being free (hence torsion-free) implies that $H_{h, t}(D; R)$ has only $c$-torsions.
\end{proof}

\begin{remark}
    The situation is apparently different when $c = 0$. In \cite{Mukherjee:ww}, it is shown that for any $2 \leq n \leq 8$ there are infinite families of links whose $\ZZ$-Khovanov homology contains $\ZZ_n$-summands.
\end{remark}

Finally we state the variance of the $\ca$-classes under the Reidemeister moves. First we define:

\begin{definition}[$\cb$-cycle]
    For any alternative orientation $o$ of $D$, define 
    \[
        \cb(D, o) = \ca(D, -o)
    \] 
    where $-o$ is the reversed orientation of $o$.
\end{definition}

The following proposition is a generalization of \cite[Proposition 2.3]{rasmussen2010khovanov}. The result is essential for the well-definedness of the link invariant $\s_c(L)$ in \Cref{sec:sc-def}. 

\begin{proposition} \label{prop:cacb-variance-under-rho}
    Let $(h, t)$ be a pair satisfying \Cref{cond:ab-cond1}, and let $c = \sqrt{h^2 + 4t}$. Suppose $D, D'$ are two diagrams related by a Reidemeister move. There is an isomorphism $\rho: H_{h, t}(D; R) \rightarrow H_{h, t}(D'; R)$ such that for any alternative orientation $o$ of $D$ (and the corresponding orientation $o'$ of $D'$)
	\begin{align*}
	    [\ca(D', o')] &= \epsilon c^j \rho[\ca(D, o)], \\
		[\cb(D', o')] &= \epsilon' c^j \rho[\cb(D, o)]
    \end{align*}
    with some $j \in \{0, \pm1 \}$ and $\epsilon, \epsilon' \in \{\pm1\}$ satisfying $\epsilon \epsilon' = (-1)^j$. (Here $c$ is not necessarily invertible, so the equation $z = c^j w$ is to be understood as $c^{-j} z = w$ when $j < 0$.) Moreover $j$ is determined as in \Cref{table:RM-k-corresp} by the type of the move and the difference of the numbers of Seifert circles (with respect to $o, o'$).
    \begin{table}[ht]
        \centering
        \begin{tabular}{lr|r}
             \toprule
             Type & $\Delta r$ & $j$ \\
             \midrule
             RM1$_L$ & 1  & 0  \\
             RM1$_R$ & 1  & 1 \\
             \midrule
             RM2     & 0  & 0  \\
                     & 2  & 1 \\
             \midrule
             RM3     & 0  & 0  \\
                     & 2  & 1  \\
                     & -2 & -1 \\
             \bottomrule
        \end{tabular}
        \caption{Exponent of $c$ corresponding to the Reidemeister moves}
        \label{table:RM-k-corresp}
    \end{table}
\end{proposition}

We give a detailed proof in \Cref{sec:canon-classes-and-RMs}. Here we only state that $\rho$ is the isomorphism constructed for the proof of \Cref{thm:t-inv}. From \Cref{table:RM-k-corresp}, we see that the exponent $j$ can be expressed by a single equation:

\begin{corollary} \label{cor:rw-2j}
    \[ j = \frac{\Delta r - \Delta w}{2} \]
    where $w$ denotes the writhe, $r$ denotes the number of Seifert circles, and the prefixed $\Delta$ denotes the difference of the corresponding values of $D$ and $D'$.
\end{corollary}

From \Cref{prop:ab-gen} and \ref{prop:cacb-variance-under-rho}, we conclude that the situation is completely analogous to $\QQ$-Lee theory when $c$ is invertible: the $\ca$-classes form a basis of $H_{h, t}(D; R)$ and are invariant (up to unit) under the Reidemeister moves.
    \subsection{Reduction of parameters} \label{subsec:H_ht-rel}

\begin{definition}
    Let $A$ be a Frobenius algebra over $R$, and $\theta$ be an invertible element in $A$. The \textit{twist} of $A$ by $\theta$ is another Frobenius algebra $(A, m, \iota, \Delta', \epsilon')$ with the same algebra structure as $A$, but with a different coalgebra structure given by:
    \[
        \Delta'(x) = \Delta(\theta^{-1}x),
        \quad 
        \epsilon'(x) = \epsilon(\theta x).
    \]
\end{definition}

\begin{lemma} [{\cite[Proposition 3]{khovanov2004}}] \label{lem:twist}
    Let $A$ be a commutative Frobenius algebra, and $A'$ be the twist of $A$ by $\theta$. For any link diagram $D$, there is an isomorphism between the chain complexes $C_A(D)$ and $C_{A'}(D)$. \qed
\end{lemma}


\begin{lemma} \label{lem:frob-alg-isom}
    Let $(h, t)$, $(h', t')$ be two pairs satisfying \Cref{cond:ab-cond1}. Let $c = \sqrt{h^2 + 4t}$ and $c' = \sqrt{h'^2 + 4t'}$. If $c = \theta c'$ for some invertible $\theta \in R$, then there is a Frobenius algebra isomorphism from $A_{h, t}$ to another Frobenius algebra $B$ such that its twist by $\theta$ gives $A_{h', t'}$.
    \begin{figure}[H]
        \centering
	    \begin{tikzcd}
            A_{h, t} \arrow{rr}{\isom} & & B \arrow[dashed]{rr}{\theta-twist} & & A_{h', t'}
        \end{tikzcd}
    \end{figure}
	\noindent
    These maps satisfy the cocycle condition: For any three pairs such that the following three arrows exist, the diagram commutes.
	\begin{figure}[H]
	    \centering
        \begin{tikzcd}
            {A_{h, t}} \arrow[rr] \arrow[rd] &  & {A_{h'', t''}} \\
             & {A_{h', t'}} \arrow[ru] & 
        \end{tikzcd}
    \end{figure}
\end{lemma}

\begin{proof}\ 
    Let $B$ be the $\theta^{-1}$-twist of $A_{h', t'}$. Define a ring homomorphism $f: R[X] \rightarrow R[X]$ by
    \[
        X\ \longmapsto\ \theta (X - u') + u.
    \]
    This descends to $f: A_{h, t} \rightarrow B$, and it can be shown that it is a Frobenius algebra isomorphism. The cocycle condition is also obvious.
\end{proof}

\begin{proposition} \label{prop:ht-relation} 
    Suppose the assumption of \Cref{lem:frob-alg-isom} holds. Then for any link diagram $D$, there is an isomorphism from $C_{h, t}(D; R)$ to $C_{h', t'}(D; R)$ under which each $\ca$-cycle in $C_{h, t}(D; R)$ is mapped to the corresponding $\ca$-cycle in $C_{h', t'}(D; R)$, multiplied by a power of $\theta$. In particular if $\theta = 1$, then the $\ca$-cycles correspond one-to-one.
\end{proposition}

\begin{proof}
    That the chain complexes are isomorphic is immediate from \Cref{lem:twist} and \ref{lem:frob-alg-isom}. Both $f$ and the chain map induced from a $\theta$-twist map any $\ca$-cycle in $C_{h, t}(D; R)$ to the corresponding $\ca$-cycle in $C_{h', t'}(D; R)$ multiplied by a power of $\theta$.
\end{proof}

\begin{corollary} \label{cor:H-normal-form}
    Suppose $(h, t)$ satisfies \Cref{cond:ab-cond1}. Let $c = \sqrt{h^2 + 4t}$. For any link diagram $D$, 
	\begin{enumerate}
	\item 
	    $C_{h, t}(D; R) \isom C_{c, 0}(D; R)$.
	    
	\item
	    $C_{h, t}(D; R) \isom C_{0, (c/2)^2}(D; R),\ $ if $c/2 \in R$\ \ (or equivalently $h/2 \in R$).
	\end{enumerate}
	In both cases, the $\ca$-cycles correspond one-to-one. \qed
\end{corollary}

\begin{proposition} \label{prop:rho-f-commutes}
    Let $(h, t),\ (h', t')$ be pairs satisfying \Cref{cond:ab-cond1} with $c = \sqrt{h^2 + 4t} = \sqrt{h'^2 + 4t'}$. For any two diagrams $D, D'$ related by a single Reidemeister move, the following diagram commutes:
    \begin{figure}[H]
        \centering
        \begin{tikzcd}
            H_{h, t}(D) \arrow{r}{\rho} \arrow{d}{f} & 
            H_{h, t}(D') \arrow{d}{f'} \\
            H_{h', t'}(D) \arrow{r}{\rho'} & 
            H_{h', t'}(D')
        \end{tikzcd}
    \end{figure}
    \noindent
    where $\rho, \rho'$ are the corresponding isomorphisms of \Cref{prop:cacb-variance-under-rho}, and $f, f'$ are the isomorphisms of \Cref{prop:ht-relation}.
\end{proposition}

\begin{proof}
    By direct calculation using the isomorphism $\rho$ given explicitly in \Cref{sec:canon-classes-and-RMs}.
\end{proof}

Given any $c \in R$, we define 
\[
    C_c(D; R) = \slfrac{ \bigoplus_{h, t} C_{h, t}(D; R) }{\sim}
\]
where $(h, t)$ runs over pairs satisfying $c = \sqrt{h^2 + 4t}$, and the equivalence relation $\sim$ is given by the isomorphism $f$ of \Cref{prop:ht-relation}. Denote the corresponding homology group by $H_c(D; R)$. \Cref{fig:ht-param-sp} depicts the $(h, t)$-parameter space, where each point $(h, t)$ corresponds to $H_{h, t}$ and the parabola $h^2 + 4t = c^2$ corresponds to $H_c$.

\begin{figure}[t]
    \centering
    \includegraphics[scale=0.35]{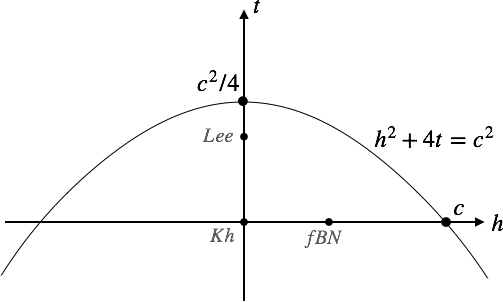}
    \caption{The $(h, t)$-parameter space.}
    \label{fig:ht-param-sp}
\end{figure}

The $\ca$-classes can be regarded as elements of $H_c(D; R)$. From \Cref{prop:rho-f-commutes} there is a well-defined isomorphism
\[
    \rho: H_c(D; R) \rightarrow H_c(D'; R).
\]
and the $\ca$-classes correspond as in \Cref{prop:cacb-variance-under-rho} under the Reidemeister moves. 
    \section{$k_c(D)$ and the invariant $\s_c(L)$} \label{sec:sc-def}

Throughout this section, we assume that $R$ is an integral domain.

\subsection{Definitions and basic properties}

\begin{definition}
	Let $M$ be an $R$-module, and $c$ be an element in $R$. Define the \textit{$c$-divisibility} of an element $z$ in $M$ by:
    \[
        k_c(z) = \max \{\ k \geq 0 \mid z \in c^k M \ \} \ \in \ [0, \infty].
    \]
\end{definition}

$\{c^k M\}_{k \geq 0}$ gives a filtration on $M$, and $k_c(z)$ gives the highest filter level that contains $z$. Note that $k_c(z) = \infty$ if $c$ is invertible or $z = 0$. 

\begin{lemma} \label{lem:c-add}
    If $n \geq 0$ then for any $z \in M$,
    \[
        k_c(c^n z) \geq k_c(z) + n.
    \]
    Moreover if $M$ is torsion-free, then the equality holds.
\end{lemma}

\begin{proof}
	$z \in c^k$ implies $c^n z \in c^{k + n}M$, so we have the inequality. Suppose $M$ is torsion free. If $k_c(c^n z)$ is infinite then so is $k_c(z)$. Suppose $k' = k_c(c^n z)$ is finite. From the maximality of $k'$ we have $n \leq k'$, and $c^n z = c^{k'} w$ for some $w \in M$. $c^n z - c^{k'} w = c^n(z - c^{k' - n}w) = 0$ implies $z = c^{k' - n}w \in c^{k' - n}M$ so $k_c(z) \geq k_c(c^n z) - n$.
\end{proof}

\begin{remark} \label{rem:k-definition}
	The equality does not hold if $M$ is not torsion free. Consider the case $R = \ZZ,\ M = \ZZ \oplus \ZZ_2,\ c = 2$ and $z = (2, 1)$. In this case $k_2(z) = 0$, but $2z = (4, 0)$ so $k_2(2z) = 2$.
\end{remark}

\begin{lemma} \label{lem:c-div}
	Let $\phi: M \rightarrow M'$ be an $R$-module homomorphism. Then for any $z \in M$,
	\[
	    k_c(z) \leq k_c(\phi(z)).
	\]
	Moreover if $\phi$ is an isomorphism, then the equality holds. 
\end{lemma}

\begin{proof}
    $z \in c^k M$ implies $\phi(z) \in \phi(c^k M) = c^k \phi(M) \subset c^k M'$.
\end{proof}

Now we return to link homology. Denote by $H_c(D; R)_f$ the \textit{free part} of $H_c(D; R)$, \text{i.e.} the quotient of $H_c(D; R)$ by its torsion submodule. By abuse of notation, we denote the image of an element $[z] \in H_c(D; R)$ by the same symbol $[z] \in H_c(D; R)_f$. 

\begin{definition}
    Let $D$ be a link diagram, and $o$ be any alternative orientation on $D$. Define
    \[
        k_c(D, o) = k_c([\ca(D, o)])
    \]
    where $[\ca(D, o)] \in H_c(D; R)_f$. We omit $o$ when it is the given orientation of $D$.
\end{definition}

Divisibility is uninteresting when it is identically $\infty$, so in the following we assume that $c$ is non-zero, non-invertible. Note that $[\ca(D, o)] \neq 0$ from \Cref{prop:ab-gen}, and when $R$ is a PID it follows that $k_c(D, o)$ is finite. In the following we only consider the given orientation, since same arguments hold for any alternative orientation by regarding it as the given one.

\begin{example}
    $k_c(\bigcirc) = 0$ since $H(\bigcirc) = C(\bigcirc) = R\<1, X\>$.
\end{example}

\begin{example}
    $D = $ (unknot with one negative crossing). Let $(R, h, t) = (\ZZ, 0, 1)$ and $c = 2$. We have
    \[
        C(D) = \{\ A \overset{\Delta}{\longrightarrow} A^{\ox 2} \ \},
        \quad
        \ca(D) = (X - 1) \otimes (X + 1).
    \]
    From $\Delta(X) = X \otimes X + 1 \otimes 1$ and $\Delta(1) = 1 \otimes X + X \otimes 1$, we see that $\ca(D)$ is homologous to $2(1 \otimes X - 1 \otimes 1)$. Since $\{ [1 \otimes X], [1 \otimes 1] \}$ form a basis of $H(D)_f$, we have $k_2(D) = 1$.
\end{example}

These examples show that $k_c(D)$ is not a link invariant. The difference of $k_c$ between two diagrams is easy to compute; it can be given without even computing the homology.

\begin{proposition} \label{prop:RM-k-relation}
    Let $D, D'$ be two diagrams of the same link. Then 
    \[
        \Delta k_c = \frac{\Delta r - \Delta w}{2},
    \]
    where $w$ denotes the writhe, $r$ denotes the number of Seifert circles, and the prefixed $\Delta$ denotes the difference of the corresponding values of $D, D'$. 
\end{proposition}

\begin{proof}
    Take any sequence of Reidemeister moves that transforms $D$ to $D'$. Let $\rho$ be the composition of the isomorphisms corresponding to the Reidemeister moves given in \Cref{prop:cacb-variance-under-rho}. Let $J$ be the sum of the $c$-exponents occurring at each move. Then $[\ca(D')] = \pm c^J \rho([\ca(D)])$. From \Cref{cor:rw-2j} we have $J = (\Delta r - \Delta w)/2$. The result follows from \Cref{lem:c-add} and \ref{lem:c-div}.
\end{proof}

Thus we obtain the following:

\begin{maintheorem}
    Let $L$ be a link. With any diagram $D$ of $L$, 
    \[
        \s_c(L) = 2k_c(D) + w(D) - r(D) + 1,
    \]
    gives an invariant of $L$.
\end{maintheorem}

First we state some basic properties of $k_c$. Obviously $k_c$ is bounded below, while it is unbounded above among diagrams of the same link, since $k_c$ increases by 1 as we add one negative twist. From the following proposition, $k_c$ may be regarded as measuring the ``non-positivity" of the diagram.

\begin{proposition} \label{prop:k_posD}
	If $D$ is positive (i.e.\ a diagram with only positive crossings), then $k_c(D) = 0$.
\end{proposition}

\begin{proof}
	The orientation preserving state of $D$ is $s = (0 \cdots 0)$. By 0-resolving the crossings one by one, we obtain a sequence of chain maps:
    \begin{align*}
    	C(D) \rightarrow C(D_0) \rightarrow \cdots \rightarrow C(D_{0 \cdots 0})
    \end{align*}
    The rightmost diagram has no crossing, so $H(D_{0 \cdots 0}) = C(D_{0 \cdots 0})$. Its $\ca$-cycle is non $c$-divisible, since it has a term $X \otimes \cdots \otimes X$ with coefficient $1$. Under the composition of the chain maps, $\ca(D)$ is mapped to $\ca(D_{0 \cdots 0})$, so from \Cref{lem:c-div} we have $k_c(D) \leq 0$.
\end{proof}

\begin{proposition} \label{prop:k-reverse}
	$$ k_c(D) = k_c(-D). $$
\end{proposition}

\begin{proof}
    Consider a Frobenius algebra automorphism on $A_{h, t}$ given by $X \mapsto -X + h$. This maps $\a, \b$ to $-\b, -\a$ respectively. The induced chain automorphism on $C_{h, t}(D)$ maps $\ca(D)$ to $\pm\cb(D) = \pm\ca(-D)$.
\end{proof}



\begin{proposition} \label{prop:k-disj-union} 
    For any two link diagrams $D, D'$,
    \[
        k_c(D \sqcup D') \geq k_c(D) + k_c(D').
    \]
    Moreover, if $R$ is a PID and $c$ is prime in $R$, then the equality holds.
\end{proposition}

\begin{proof}
    From \Cref{prop:C_ht-basic} we have $C(D \sqcup D') \isom C(D) \ox C(D')$. From a general argument of homological algebra, there is a homomorphism
	\[
        h: H(D)_f \otimes H(D')_f \longrightarrow H(D \sqcup D')_f,
    \]
    that maps $[z] \otimes [w]$ to $[z \otimes w]$, and it is an isomorphism if $R$ is a PID. Let $\ca = \ca(D),\ \ca' = \ca(D')$. The $\ca$-cycle of $D \sqcup D'$ is given by $\ca'' = \ca \otimes \ca'$. Let $[\ca] = c^k[\ca_0],\ [\ca'] = c^{k'}[\ca_0']$ with maximal $k, k'$. Then
    \[
        [\ca''] 
            = [\ca \ox \ca'] 
            = h([\ca] \ox [\ca']) 
            = c^{k + k'}h([\ca_0] \ox [\ca'_0]),
    \] 
    so the inequality holds.
    
    Now suppose $R$ is a PID and $c$ is prime in $R$. Let $[\ca''] = c^{k''}[\beta]$. We have $h([\ca_0] \otimes [\ca_0']) = c^l[\beta]$ where $l = k'' - (k + k')$. We prove that $l = 0$. Let $\{ [z_i] \}$, $\{ [z'_j] \} $ be the bases of $H(D)_f, H(D')_f$ respectively. Since $h$ is an isomorphism, $\{ [z_i \ox z'_j] \}$ is a basis of $H(D \sqcup D')_f$. Let $[\ca_0] = \sum_i a_i [z_i],\ [\ca'_0] = \sum_j a'_j [z'_j]$. Then
    \[
        h([\ca_0] \otimes [\ca_0']) 
            = \sum_{i,j} a_i a'_j [z_i \ox z'_j] 
            \ \in \ c^l H(D \sqcup D')_f.
    \]
    If $l > 0$, then $c \mid a_i a'_j$ for all $i, j$. Since $c$ is prime, one of $[\ca_0], [\ca_0']$ must be divisible by $c$. This contradicts the maximality of $k$ or $k'$.
\end{proof}



\begin{proposition}\label{prop:k-conn-sum-weak}
    For any two link diagrams $D, D'$,
    \[
        k_c(D \# D') \leq k_c(D \sqcup D') \leq k_c(D \# D') + 1.
    \]
\end{proposition}

\begin{proof}
    $D \# D'$ and $D \sqcup D'$ are related by fusion moves, and the $\ca$-cycles correspond as:
    \[
        \ca(D \# D') 
        \ \overset{\Delta}{\longmapsto} \  
        \ca(D \sqcup D')
        \ \overset{m}{\longmapsto} \  
        \pm c\ca(D \# D').
    \]
\end{proof}

Next we state some basic properties of $\s_c$. The following properties can be obtained immediately from the previous results.

\begin{proposition} \label{prop:s-properties}
    Let $L, L'$ be any two link diagrams.
    \begin{enumerate}
        \item 
            $\s_c(\bigcirc) = 0$.
            
        \item 
            $\s_c(L) = \s_c(-L)$.
            
        \item 
            $\s_c(L \sqcup L') \geq \s_c(L) + \s_c(L') - 1$. 
            
            
        \item 
            $\s_c(L \# L') = \s_c(L \sqcup L') \pm 1.$
            
    \end{enumerate}
    If $R$ is a PID and $c$ is prime in $R$, then we have
    \begin{itemize}
        \item [3'.]
            $\s_c(L \sqcup L') = \s_c(L) + \s_c(L') - 1$. 
            
        \item [4'.]
            $\s_c(L \# L') = \s_c(L) + \s_c(L') \text{\ \ or\ \ } \s_c(L) + \s_c(L') - 2.$

    \end{itemize}
    \qed
\end{proposition}

\begin{lemma} \label{lem:chiS}
    Let $L$ be a link, $D$ be a diagram of $L$. Let $S$ be the Seifert surface of $L$ obtained by applying the Seifert's algorithm to $D$. Then
    \[
        \chi(S) = 2 - 2g(S) - |L| = r(D) - n(D),
    \]
    where $\chi$ is the Euler number, and $g$ is the genus.
    \qed
\end{lemma}

\begin{proposition}
    $$ \s_c(L) \equiv |L| - 1 \bmod{2}. $$
    \qed
\end{proposition}

\begin{proposition} \label{prop:s-pos-link}
    Let $L$ be a positive link, and $D$ be a positive diagram of $L$. Let $S$ be the Seifert surface of $L$ obtained by applying the Seifert's algorithm to $D$. Then 
    \[
        \s_c(L) = 2g(S) + |L| - 1.
    \]
    In particular for a positive knot $K$,
    \[
        \s_c(K) = 2g(S).
    \]
    \qed
\end{proposition}

    \subsection{Behavior under cobordisms} \label{subsec:ca-cobordism}

Let $L, L'$ be two links in $\RR^3$, and $S \subset \RR^3 \times [0, 1]$ be an (oriented smooth) cobordism between $L$ and $L'$ with $\partial S = (-L) \times \{0\} \cup L' \times \{1\}$. Let $D, D'$ be diagrams of $L, L'$ respectively. Following \cite[Section 6.3]{khovanov2000} and \cite[Section 4]{rasmussen2010khovanov}, we construct a homomorphism 
\[
    \phi: H_c(D; R) \rightarrow H_c(D'; R).
\]

By modifying $S$ by a small isotopy, we may assume that $S$ decomposes as a union of elementary cobordisms, and that except for finite many $t$'s the section $S \cap (\RR^3 \times \{t\})$ is a link and also the corresponding diagram is regular. Decompose $S = \bigcup_{i=0}^{N-1} T_i$ so that each $T_i = S \cap (\RR^3 \times [t_i, t_{i + 1}])$ corresponds to a Reidemeister move or a Morse move. Let $L_i$ = $S \cap (\RR^3 \times \{t_i\})$ with $L_0 = L,\ L_N = L'$, and $D_i = p(L_i)$ with $D = D_0,\ D' = D_N$.

Each $T_i$ gives a homomorphism $\phi_i: H_c(D_i) \rightarrow H_c(D_{i+1})$. Namely, if $T_i$ corresponds to a Reidemeister move then $\phi_i$ is the isomorphism $\rho$ given in \Cref{prop:cacb-variance-under-rho}. If $T_i$ corresponds to a Morse move, then $\phi_i$ is induced from the chain map given by the corresponding operation of the Frobenius algebra $A$. Let $\phi$ be the composition of all $\phi_i$'s. The following is a generalization of \cite[Proposition 4.1]{rasmussen2010khovanov} and \cite[Proposition 3.2]{Rasmussen:2005vo}.

\begin{proposition} \label{prop:k-cobordism}
    In addition to the above setting, assume every component of $S$ has a boundary in $L$. Then the induced homomorphism
    \[
        \phi: H_c(D; R)_f \rightarrow H_c(D'; R)_f
    \]
    maps 
    \[
        \phi[\ca(D)] = \pm c^l [\ca(D')],\quad 
        l = \frac{1}{2}(-\Delta r + \Delta w - \chi(S)),
    \]
    where the prefixed $\Delta$ denotes the difference of the corresponding values of $D, D'$.
\end{proposition}

\begin{figure}[ht]
    \centering
    \includegraphics[scale=0.4]{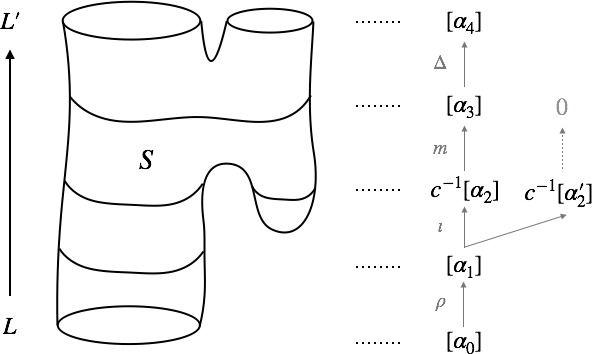}
    \caption{Cobordism map}
    \label{fig:cob-map}
\end{figure}

\begin{proof}
    First, we may assume that $c$ is invertible, since $H_c(D; R)_f \rightarrow H_c(D; R_c)$ is injective from \Cref{cor:H(D)-tor}. Let $S_i = S \cap (\RR^3 \times [0, t_i])$. Within the alternative orientations of $S_i$ (i.e.\ possible orientations on the underlying unoriented surface), there are ones that agree with the orientation of $L$ on the bottom boundary. We call such orientations to be \textit{permissible}. We also call an alternative orientation on $L_i$ to be \textit{permissible} if it is induced from a permissible orientation on $S_i$. Permissible orientations on $S_i$ and permissible orientations on $L_i$ correspond one-to-one, so we identify the two. Indeed, suppose $o, o'$ are permissible orientations that differ on a component $T$ of $S_i$ but induce the same orientation on $L_i$. First $T$ has no boundary in $L$, otherwise $o, o'$ must be equal on $T$. Similarly, $T$ has no boundary in $L_i$. This implies that $T$ is a closed component, which contradicts the assumption. 
	
	\begin{claim}
    	For each permissible orientation $o$ on $L_i$,
	    \[
	        \phi_i[\ca(D_i, o)] = \pm c^j\sum_{o'}[\ca(D_{i + 1}, o')]
	   \]
	   where $j \in \{0, \pm 1\}$, and the sum runs over a (possibly empty) set of permissible orientations of $L_{i + 1}$ that extend $o$.
	\end{claim}
	
	For the Reidemeister moves, $o$ extends uniquely to a permissible orientation $o'$ and from \Cref{prop:cacb-variance-under-rho} we know the claim is true. For the 0-handle move, there are two permissible orientations that extend $o$, and the corresponding $\ca$-classes are $[\ca(D_i, o) \otimes \a]$ and $[\ca(D_i, o) \otimes \b]$. Since $1 = c^{-1}(\a - \b)$, the claim holds with $j = -1$. For the 1-handle move there are several cases to consider.
	(\rn{1}) Suppose the move splits a component of $L_i$. Then $o$ extends uniquely to a permissible orientation $o'$. If the move splits one of the Seifert circles (with respect to $o$) then $[\ca(D_i, o)]$ is mapped to $[\ca(D_{i + 1}, o')]$. If it merges two circles then it is mapped to $\pm c[\ca(D_{i + 1}, o')]$. (Note that the Seifert circle(s) may split or merge, regardless of the splitting of the link.)
	(\rn{2}) Suppose the move merges two components of $L_i$. If the orientations on the two components are coherent with respect to the handle, then the situation is similar to (\rn{1}). Otherwise $[\ca(D_i, o)]$ is mapped to $0$, since the two arcs where the handle is attached have the same direction and are colored differently. 
	Finally for the 2-handle move, $o$ extends uniquely to a permissible orientation $o'$. Since $\epsilon(\a) = \epsilon(\b) = 1$, we have $\phi_i[\ca(D_i, o)] = [\ca(D_{i + 1}, o')]$.
	
	\begin{claim} \label{cl:phi_i-image}
	    Suppose $x \in H(D_i)_f$ is of the form
	    \[
	        x = \sum_o (\pm c^{k_o})[\ca(D_i, o)]
	        \quad 
	        (k_o \in \ZZ)
	    \]
        where $o$ runs over a set of permissible orientations of $S_i$. Then the image of $x$ under $\phi_i$ also has the same form (possibly zero).
	\end{claim}
	
	This is obvious from the previous claim and the fact that no two permissible orientations extend to the same one.
	
	\begin{claim} 
	    Regarding the given orientations on $D$ and $D'$,
	    \[
            \phi[\ca(D)] = \pm c^l [\ca(D')]
        \]
        for some integer $l$.
	\end{claim}
	
	We see that the successive images of $[\ca(D)]$ are of the form of \Cref{cl:phi_i-image}. At each level $[\ca(D_i)]$ is mapped to $\pm c^j[\ca(D_{i+1})]$ (modulo other terms), since mapping to $0$ can happen only when a 1-handle merges inconsistently oriented components. At the end, there is only one permissible orientation on $S_N = S$, that is the given orientation of $S$. Thus $[\ca(D)]$ is mapped to some $c$-power multiple of $[\ca(D')]$. (The right side of \Cref{fig:cob-map} depicts the successive images under each $\phi_i$.)
	
	\medskip
	\noindent \textit{Proof continued.} 
    Now it remains to describe $l$. From the discussion of the previous claim, we see that $l$ is given by the sum of the $c$-exponents appearing in $\phi_i[\ca(D_i)]$. Let $n_0, n_1, n_2$ be the numbers of 0-, 1-, 2-handle moves respectively. Also let $n_1 = n_{1, m} + n_{1, s}$, where $n_{1, m}$ ($n_{1, s}$, resp.) is the number of times the Seifert circles of $D_i$ are merged (splitted, resp.) by the 1-handle move. Let $J$ be the sum of $c$-exponents for the Reidemeister moves. For the Morse moves, $j = -1$ occurs only by the 0-handle moves, and $j = +1$ occurs only by the 1-handle moves that the Seifert circles merge. Thus we have 
    \[
        l = J - n_0 + n_{1, m}.
    \]
    
    Let $\Delta r = \Delta r_R + \Delta r_M$, where $\Delta r_R$ (resp. $\Delta r_M$) is the sum of the differences of $r$ at each step corresponding to the Reidemeister move (resp. Morse move). Obviously
    \[
        \Delta r_M = n_0 - n_{1, m} + n_{1, s} - n_2.
    \]
    $w$ is constant under the Morse moves, and from \Cref{cor:rw-2j} we have
    \[
        -J = \frac{\Delta r_R - \Delta w}{2}.
    \]
    Thus
    \begin{align*}
        l 
            &= -\frac{1}{2}(\Delta r_R - \Delta w) - n_0 + n_{1, m} \\
            &= \frac{1}{2}(-\Delta r + \Delta r_M + \Delta w - 2n_0 + 2n_{1, m}) \\
            &= \frac{1}{2}(-\Delta r + \Delta w - n_0 + (n_{1, m} + n_{1, s}) - n_2) \\
            &= \frac{1}{2}(-\Delta r + \Delta w - \chi(S)).
    \end{align*}
    \setcounter{claim}{0}
\end{proof}

\begin{proposition} \label{prop:hat-s-cobordism}
    With the assumption of \Cref{prop:k-cobordism}, we have
    \[
        \s_c(L') - \s_c(L) \geq \chi(S).
    \]
    If also every component of $S$ has a boundary in both $L$ and $L'$, then
    \[
        |\s_c(L') - \s_c(L)| \leq -\chi(S).
    \]
    \qed
\end{proposition} 

\begin{remark}
    In the latter case $\chi(S) \leq 0$ since each component $T$ of $S$ has at least two boundary components and $\chi(T) \leq 0$.
\end{remark}

From \Cref{prop:hat-s-cobordism}, we obtain properties of $\s_c$ that are common to $s$. We can also reprove the Milnor conjecture using $\s_c$. 

\begin{maintheorem} \label{thm:s_c-properties}
    The following holds:
    \begin{enumerate}
        \item $\s_c$ is a link concordance invariant,
        \item $|\s_c(K)| \leq 2g_*(K)$ for any knot $K$, and
        \item $\s_c(K) = 2g_*(K) = 2g(K)$ if $K$ is positive,
    \end{enumerate}
    where $g_*$ is the slice genus, and $g$ is the ordinary genus of a knot.
    \qed
\end{maintheorem}

\begin{corollary}[The Milnor Conjecture \cite{Milnor1968}]
    The slice genus and the unknotting number of the $(p, q)$ torus knot are both equal to $(p - 1)(q - 1)/2$.
    \qed
\end{corollary}


We state some more properties of $k_c$ and $\s_c$ that follows from \Cref{prop:hat-s-cobordism}.

\begin{proposition} \label{prop:s-link-mirror}
    \begin{itemize}
        \item For any link $L$, 
            \[ 1 - 2|L| \leq \s_c(L \sqcup \bar{L}) \leq 1. \]
        \item For any link diagram $D$,
            \[ r(D) - |D| \leq k_c(D \sqcup \bar{D}) \leq r(D). \]
    \end{itemize}
\end{proposition}

\begin{proof}
    There is a cobordism consisting of $|L|$ saddles connecting $L \sqcup \bar{L}$ to the $|L|$-component unlink.
\end{proof}

\begin{corollary}
    Suppose $R$ is a PID and $c$ is prime in $R$. Then:
    \begin{itemize}
        \item For any link $L$, 
            \[ 2 - 2 |L| \leq \s_c(L) + \s_c(\bar{L}) \leq 2. \]
        \item For any link diagram $D$,
            \[ r(D) - |D| \leq k_c(D) + k_c(\bar{D}) \leq r(D). \]
    \end{itemize}
    \qed
\end{corollary}

\begin{proposition} \label{prop:cross-rm-bound}
    Let $D$ be a link diagram. Let $D'$ be the diagram obtained from $D$ by removing one crossing in the orientation preserving way. If the crossing is positive, then
    \[
        k_c(D) \leq k_c(D') \leq k_c(D) + 1.
    \]
    If it is negative, then
    \[
        k_c(D) - 1 \leq k_c(D') \leq k_c(D).
    \]
\end{proposition}

\begin{proof}
    Removing a crossing can be realized by attaching a 1-handle near the crossing.
\end{proof}

\begin{corollary} \label{cor:k_c-bound}
	For any link diagram $D$,
	\[
	    0 \leq k_c(D) \leq n^-(D)
	\]
	where $n^-(D)$ is the number of negative crossings of $D$. \qed
\end{corollary}
    \subsection{Coincidence with $s$}
\label{subsec:coincidence-with-s}

In this section we restrict to knots, and to $(R, c) = (F[h], h)$ where $F$ is a field of $\fchar{F} \neq 2$ and $h$ is an indeterminate of degree $-2$. Let $R_0 = F[h^2] \subset R$. Note that both $R$ and $R_0$ are PIDs. Let $D$ be a knot diagram. We denote
\begin{align*}
    C(D; R_0) &= C_{0, (h/2)^2}(D; R_0), \\
    C(D; R) &= C_{0, (h/2)^2}(D; R),
\end{align*}
and the corresponding homology groups by $H(D; R_0),\ H(D; R)$ respectively. Note that $H(D; R)$ is naturally isomorphic to $H_\BN(D; R) = H_{h, 0}(D; R)$ from \Cref{cor:H-normal-form}. Since $H(D; R_0)_f$ is torsion-free, the natural map
\[
    H(D; R_0)_f \longrightarrow H(D; R_0)_f \otimes_{R_0} R
\]
is injective. Also since $R$ is flat over $R_0$ (for $R$ is torsion-free and $R_0$ is a PID), we have 
\[
    H(D; R_0)_f \otimes R \isom (H(D; R_0) \otimes R)_f \isom H(D; R)_f.
\]
Under this correspondence we regard $H(D; R_0)_f \subset H(D; R)_f$. Let $(\ca, \cb) = (\ca(D), \cb(D))$. Note that $\ca, \cb \notin C(D; R_0)$. The following two lemmas are analogous to \cite[Lemma 3.5]{rasmussen2010khovanov}.

\begin{lemma} \label{lem:k-mirror-lem2}
    $C(D; R_0)$ decomposes into a direct sum of two subcomplexes:
    \begin{align*}
        C(D; R_0)'  &= \Span_{R_0} \{\ x \mid \qdeg(x) \equiv |D| \bmod{4}\ \} \\
        C(D; R_0)'' &= \Span_{R_0} \{\ x \mid \qdeg(x) \equiv |D| + 2 \bmod{4}\ \}.
    \end{align*}
    where $x$ runs over the enhanced states of $D$. \qed
\end{lemma}

\begin{lemma} \label{lem:k-mirror-lem3}
    Two elements
    \begin{align*}
        \cx &= 
            (\ca + \cb) / 2,
        \\
        \cy &= 
            (\ca - \cb) / h.
    \end{align*}
    belong to $C(D; R_0)$. Either one is contained in $C(D; R_0)'$ and the other is in $C(D; R_0)''$. 
\end{lemma}

\begin{proof}
    Recall $\a = X + (h/2),\ \b = X - (h/2)$. By expanding $\ca$ into linear combinations of enhanced states, we see that $\cx = (\ca + \cb)/2$ is the part of $\ca$ with even numbers of $(h/2)$'s in its tensor factors, and $(\ca - \cb)/2 = (h/2) \cy$ is the odd part. Both $\cx, \cy$ belong to $C(D; R_0)$, since the coefficients are powers of $t = (h/2)^2 \in R_0$. The two elements live separately in the summands, since $\cx$ (resp.\ $\cy$) has even (resp.\ odd) number of $1$'s in each of its tensor factors.
\end{proof}

Let $A = A_{0, (h/2)^2}$. We endow an $A$-module structure on $C(D; R)$ following \cite{Khovanov:2002wo}. Take any point $p$ on an arc of $D$, and a small circle $\bigcirc$ near $p$. Merging $\bigcirc$ into a neighbourhood of $p$ corresponds to the multiplication
\[
    m_p: A \ox C(D) \longrightarrow C(D).
\]
Define an endomorphism $X_p$ by
\[
    X_p = m_p(X \ox -) : C(D) \rightarrow C(D).
\]

We can prove the following by tracing the proofs of \cite[Lemma 2.3]{Hedden:2012iz}, \cite[Lemma 2.1]{Alishahi:2017ug} and \cite[Lemma 3.3]{Alishahi:2017wl}.

\begin{lemma} \label{lem:X_p endo}
    If $p, q$ are two marked points on a strand of $D$ separated by a crossing $x$, then $X_p$ and $-X_q$ are chain homotopic. \qed
\end{lemma}

Thus we obtain an endomorphism $X$ by 
\[
    X \cdot x = 
    \begin{cases}
        \ \ X_p(x) 
            & \text{if the arc containing $p$ is colored $\a$,} \\
        -X_p(x) 
            & \text{otherwise.}
    \end{cases}
\]
$X$ is defined independent of the choice of the point $p$. This gives an $R[X]$-module structure on $C(D; R)$, and obviously it descends to an $A$-module structure.

\begin{lemma}
    $X$ maps $C(D; R_0)'$ to $C(D; R_0)''$ and vice versa.
\end{lemma}

\begin{proof}
    Obvious since $X_p$ is a map of q-degree $-2$.
\end{proof}

\begin{lemma} \label{lem:X_p}
    $X$ maps:
    \begin{align*}
        \ca &\longmapsto \ \ (h/2)\ca, \\
        \cb &\longmapsto - (h/2)\cb, \\
        \cx &\longmapsto (h/2)^2\cy, \\
        \cy &\longmapsto \cx.
    \end{align*}
\end{lemma}

\begin{proof}
	Suppose the arc containing $p$ is colored $\a$. We may write $\ca = \a \otimes \ca'$ and $\cb = \b \otimes \cb'$ for some $\ca', \cb'$. Then 
   	\begin{align*}
   		\ca = 
   		    (X + h/2) \otimes \ca'\ 
   		    &\overset{X_p}{\longmapsto}\ 
   		    ((h/2)^2 + (h/2)X) \otimes \ca' 
   		    = (h/2)\ca \\
   		\cb = 
   		    (X - h/2) \otimes \cb'\ 
   		    &\longmapsto\ 
   		    ((h/2)^2 - (h/2)X) \otimes \cb' 
   		    = - (h/2)\cb.
   	\end{align*}
    The result is similar for the other case. The images of $\cx, \cy$ are obvious from definition.
\end{proof}

\begin{proposition} \label{prop:H(D)-basis}
    There is a unique class $[\cz] \in H(D; R_0)_f$ such that:
    \begin{itemize}
        \item 
            $\{\ [\cz], X [\cz]\ \}$ is a basis of $H(D; R_0)_f$ and of $H(D; R)_f$, and,
            
        \item 
            $[\ca], [\cb] \in H(D; R)_f$ can be written as
           	\begin{align*}
           		[\ca] &= \quad\ h^k(\ X[\cz] + (h/2)[\cz] \ ) \\
           		[\cb] &= (-h)^k(\ X[\cz] - (h/2)[\cz]\ )
           	\end{align*}
            where $k = k_h(D)$.
    \end{itemize}
\end{proposition}

\begin{proof}
    From \Cref{lem:k-mirror-lem2} we have $H(D; R_0) = H(D; R_0)' \oplus H(D; R_0)'' \isom (R_0)^2$. From \Cref{lem:k-mirror-lem3} each summand is isomorphic to $R_0$, so there is a basis $\{ [\cx_0], [\cy_0] \}$ of $H(D; R_0)_f$ such that \[
        [\cx] = x'[\cx_0],\quad 
        [\cy] = y'[\cy_0]
    \]
    for some $x', y' \in R_0$, and
    \begin{align*}
        [\ca] = [\cx] + (h/2)[\cy] = x'[\cx_0] + (h/2)y'[\cy_0], \\
        [\cb] = [\cx] - (h/2)[\cy] = x'[\cx_0] - (h/2)y'[\cy_0].
    \end{align*}
    
    We have $H(D; R_0)_f \otimes R \isom H(D; R)_f$ so $\{ [\cx_0], [\cy_0] \}$ is also a basis of $H(D; R)_f$. From the definition of $k$,
    \[
        [\ca], [\cb] \in h^k H(D; R)_f = R\< h^k[\cx_0], h^k[\cy_0] \>.
    \]
    From 
    \[
        [\cx] = ([\ca] + [\cb])/2,\quad
        (h/2)[\cy] = ([\ca] - [\cb])/2,
    \]
    there are $x, y \in R$ such that
    \[
        x' = h^k x, \quad (h/2)y' = h^k y.
        \tag{1} \label{eq:1}
    \]
    and
   	\begin{align*}
   		[\ca] &= h^k(\ x[\cx_0] + y[\cy_0] \ ) \\
   		[\cb] &= h^k(\ x[\cx_0] - y[\cy_0]\ ).
   	\end{align*}
   	
    First, $x$ and $y$ are not commonly divisible by $h$ from the maximality of $k$. With the endomorphism $X$, we have:
    \begin{align*}
        [\cx] = h^k x[\cx_0] \ 
            &\overset{X}{\longmapsto} \ 
            (h/2)^2[\cy] 
            = (h/2) h^k y[\cy_0] 
            = h^k x X[\cx_0] \\
        (h/2)[\cy] = h^k y[\cy_0] \ 
            &\overset{X}{\longmapsto} \ 
            (h/2)[\cx] 
            = (h/2) h^k x[\cx_0] 
            = h^k y X[\cy_0]
   	\end{align*}
   	so
    \begin{align*}
        (h/2) y[\cy_0] &= x X[\cx_0]
        \tag{2} \label{eq:A}
        \\
        (h/2) x[\cx_0] &= y X[\cy_0]
        \tag{3} \label{eq:B}.
   	\end{align*}
   	Since $[\cx_0], [\cy_0]$ form a basis, this implies
    \[ 
        x \mid (h/2)y, \quad y \mid (h/2)x.
    \]
    Together with \eqref{eq:1} we have 
    \[
        \begin{cases}
            x \sim 1, \ y \sim h/2 & \text{if $k$ is even,}\\
            x \sim h/2, \ y \sim 1 & \text{if $k$ is odd}
        \end{cases}
    \]
    Thus we may define
    \[
        [\cz] = \begin{cases}
            (h/2)^{-1}y[\cy_0] & \text{if $k$ is even} \\
            (h/2)^{-1}x[\cx_0] & \text{if $k$ is odd}.
        \end{cases}
    \]
    We check that $[\cz]$ satisfies the required conditions. If $k$ is even, then from \eqref{eq:B} we have
    \[
        x [\cx_0] = (h/2)^{-1}y X[\cy_0] = X[\cz].
    \]
    Thus $[\cz], X[\cz]$ are associated to $[\cy_0], [\cx_0]$ respectively, and
   	\begin{align*}
   		[\ca] &= h^k(\ X[\cz] + (h/2)[\cz]\ ), \\
   		[\cb] &= h^k(\ X[\cz] - (h/2)[\cz]\ ).
   	\end{align*}
   	
    Similarly if $k$ is odd, then $[\cz], X[\cz]$ are associated to $[\cx_0], [\cy_0]$ respectively, and 
   	\begin{align*}
   		[\ca] &= h^k(\ (h/2)[\cz] + X[\cz] \ ), \\
   		[\cb] &= h^k(\ (h/2)[\cz] - X[\cz]\ ).
   	\end{align*}
   	
    Hence in both cases $\{\ [\cz], X [\cz]\ \}$ form a basis of $H(D; R_0)_f$ and of $H(D; R)_f$, and the desired descriptions of $[\ca], [\cb]$ hold. Uniqueness follows by comparing the descriptions of $[\ca], [\cb]$.
\end{proof}
\setcounter{case}{0}

There is unimodular pairing
\[
    \< -, - \>: C(D) \otimes C(\bar{D}) \longrightarrow R
\]
defined by the composition of the isomorphism $T: C(\bar{D}) \rightarrow C(D)^*$ of \Cref{prop:C_ht-basic} and the standard pairing between $C(D)$ and $C(D)^*$. From a general argument of homological algebra, this descends to 
\[
    \< -, - \>: H(D)_f \otimes H(\bar{D})_f \longrightarrow R,
\]
and is unimodular since $R$ is a PID.

\begin{notation}
    We write
    \[
        \<
            \begin{pmatrix}
                a \\ b
            \end{pmatrix}
            ,\ 
            \begin{pmatrix}
                c & d
            \end{pmatrix}
        \>
        =
        \begin{pmatrix}
            \<a, c \> & \<a, d \> \\
            \<b, c \> & \<b, d \> \\
        \end{pmatrix}.
    \]
\end{notation}

\begin{lemma}
\label{lem:cacb-pairing} Let $(\ca, \cb) = (\ca(D), \cb(D))$ and $(\bar{\ca}, \bar{\cb}) = (\ca(\bar{D}), \cb(\bar{D})$. Then
    \[
        \<
            \begin{pmatrix}
                \ca \\ \cb
            \end{pmatrix}
            ,\ 
            \begin{pmatrix}
                \bar{\ca} & \bar{\cb}
            \end{pmatrix}
        \>
        =
        h^{r(D)}
        \begin{pmatrix}
            (-1)^b & 0 \\
            0 & (-1)^a  \\
        \end{pmatrix}
    \]
    where $a, b$ are the numbers of $\a$'s and $\b$'s in the tensor factors of $\ca$.
\end{lemma}

\begin{proof}
    From 
    \begin{gather*}
        \a = X + (h/2)1,\quad \b = X - (h/2)1 \\
        T(\bar{\a}) = 1^* + (h/2)X^*,\quad T(\bar{\b}) = 1^* - (h/2)X^*
    \end{gather*}
    we have:
    \begin{align*}
        \<
            \begin{pmatrix}
                \a \\ \b
            \end{pmatrix}
            ,\ 
            \begin{pmatrix}
                \bar{\a} & \bar{\b}
            \end{pmatrix}
        \>
		=
		\begin{pmatrix}
		    c & 0 \\
		    0 & -c
		\end{pmatrix}
    \end{align*}
    The result follows from the definition of the $\ca$-cycles.
\end{proof}

\begin{proposition}[Mirror formula] \label{prop:k-mirror}
    \[
        k_h(D) + k_h(\bar{D}) = r(D) - 1.
    \]
\end{proposition}

\begin{proof}
    With the description of \Cref{prop:H(D)-basis}, 
    \begin{gather*}
        \<
            \begin{pmatrix}
                \ca \\ \cb
            \end{pmatrix}
            ,\ 
            \begin{pmatrix}
                \bar{\ca} & \bar{\cb}
            \end{pmatrix}
        \> \\
        =
        h^{k + k'}
        \begin{pmatrix}
            h/2 & 1 \\
            \mp h/2 & \pm 1
        \end{pmatrix}
        \<
            \begin{pmatrix}
                \cz \\ X\cz
            \end{pmatrix}
            ,\ 
            \begin{pmatrix}
                \bar{\cz}, X\bar{\cz}
            \end{pmatrix}
        \>
        \begin{pmatrix}
            h/2 & \mp h/2 \\
            1 & \pm 1
        \end{pmatrix}
    \end{gather*}
    Since the pairing is unimodular, the middle matrix on the right hand side must have unital determinant. Together with \Cref{lem:cacb-pairing}, by comparing the determinants on both sides we have
    \[
        2r(D) = 2(k + k') + 2.
    \]
\end{proof}

\begin{corollary} \label{cor:k_h-neg}
    For a negative knot diagram $D$,
    \[
        k_h(D) = r(D) - 1.
    \]
    \qed
\end{corollary}

\begin{proposition} \label{prop:k-conn-sum}
    Let $D, D'$ be knot diagrams.
    \[
        k_h(D \# D') = k_h(D) + k_h(D'). 
    \]
\end{proposition}

\begin{proof}
    Since $R$ is a PID, from \Cref{prop:k-disj-union}, \ref{prop:k-conn-sum-weak}, we have
    \[
        k_h(D \# D') \leq k_h(D) + k_h(D').
    \]
    With \Cref{prop:k-mirror} we have 
    \begin{align*}
        k_h(D \# D') 
            &= -k_h(\bar{D \# D'}) + r(\bar{D \# D'}) - 1 \\
            &\geq -(k_h(\bar{D}) + k_h(\bar{D'})) + (r(D) + r(D') - 1) - 1 \\
            &= k_h(D) + k_h(D').
    \end{align*}
\end{proof}

From \Cref{prop:k-mirror}, \ref{prop:k-conn-sum} we obtain:

\begin{proposition} \label{thm:sc-hom}
    $\s_h$ defines a homomorphism from the concordance group of knots in $S^3$ to $2\ZZ$.
    \qed
\end{proposition}

\begin{remark}
    The above arguments also hold for $(R, c) = (\ZZ, 2)$, with a little modification in the proof of \Cref{prop:H(D)-basis} using $H_c(D; \FF_p) \isom \FF_p^2$ for $p \geq 3$. 
\end{remark}

Finally we prove that $\s_h$ coincides with Rasmussen's $s$-invariant over $F$. The following definition of $s$ is given by Beliakova and Wehrli in \cite[Section 7.1]{beliakova2008categorification}.

\begin{definition} \label{def:ras-s}
	Let $F$ be a field of $\fchar{F} \neq 2$. Let $L$ be a link and $D$ be any diagram of $L$. The \textit{Rasmussen invariant} (over $F$) of a link $L$ is defined by:
    \begin{align*}
        s(L; F) = \frac{ \qdeg[\ca + \cb] + \qdeg[\ca - \cb] }{2}
    \end{align*}
	where $[\ca], [\cb]$ are the $\ca$-, $\cb$-classes of $D$ in $H_\Lee(D; F)$, and $\qdeg$ is the filtered q-degree of $H_\Lee(D; F)$.
\end{definition}

Using the fact that the q-degree of $[\ca + \cb]$ and $[\ca - \cb]$ differs by 2, and that $\qdeg[\ca] = \min\{ \qdeg[\ca + \cb], \qdeg[\ca - \cb]\}$, we can also write
\[
    s(L; F) = \qdeg[\ca] + 1.
\]

\begin{maintheorem} \label{thm:s-with-k}
    For any knot $K$, 
    \[
        s(K; F) = \s_h(K; F[h]).
    \]
\end{maintheorem} 

\begin{proof}
    Since both $s$ and $\s_h$ changes sign by mirroring the knot, it suffices to prove the inequality
    \[
        s(K; F) \geq \s_h(K; F[h]).
    \]

    Denote $C(D; F) = C_{0, 1}(D; F)$ and $C_h(D; F[h]) = C_{0, (h/2)^2}(D; F[h])$. Let $\pi: C_h(D; F[h]) \rightarrow C(D; F)$ be the chain map induced from $h \mapsto 2$. Let $\ca, \ca_h$ be the $\ca$-cycles of $D$ in $C(D; F)$, $C_h(D; F[h])$ respectively. Then $\pi(\ca_h) = \ca$. Let $[\ca_h] = h^k[\ca_h']$ with maximal $k$. 
    
    We denote the bigraded q-degree of $H_h(D; F[h])$ by $\qdeg_h$. First, $\ca_h$ is homogeneous with $\qdeg_h[\ca_h] = w(D) - r(D)$. From $\deg(h) = -2$ we have $\qdeg_h[\ca_h'] = 2k + w(D) - r(D)$. Since $\pi_*: H_h(D; F[h]) \rightarrow H(D; F)$ is q-degree non decreasing, we have
    \begin{align*}
        s(K; F) 
            &= \qdeg[\ca] + 1 \\
	        &= \qdeg(\pi_*[\ca_h]) + 1 \\
	        &= \qdeg(\pi_*[\ca_h']) + 1 \\
        	&\geq \qdeg_h[\ca_h'] + 1 \\
            &= \s_h(K; F[h]).
    \end{align*}
\end{proof}

\begin{remark}
    There is a well known lower bound for $s$ (\cite[Lemma 1.3]{Shumakovitch:2007})
    \[
        s(K) \geq w(D) - r(D) + 1,
    \]
    so we see that $2k_h(D)$ gives the correction term of the inequality.
\end{remark}

\begin{remark}
    A similar inequality
    \[
        s(L; F) \geq 2c_F(D) + w(D) - r(D) +1,
    \]
    is given by Collari in \cite{Collari:2017bennequin}, where $c_F(D)$ is a transverse link invariant called the \textit{c-invariant}. $c_F(D)$ is defined by the $h$-divisibility of $[\ca(D)]$ in Bar-Natan homology without discarding the torsions (see \Cref{rem:k-definition} and \Cref{sec:further}).
\end{remark}

Finally, we relate our results to some other alternative definitions of $s$. The following one is given by Khovanov in \cite{khovanov2004} using the bigraded version of Lee homology.

\begin{corollary} \label{cor:H(D; F[h])-module-struct}
    Let $t$ be a formal variable of degree $-4$. $H_{0, t}(K; F[t])_f$ admits a free $F[X]$-module structure, and there is a bigrading preserving isomorphism
    \[
        H_{0, t}(K; F[t])_f \isom (F[X])[0, s(K; F) + 1].
    \]
\end{corollary}

\begin{proof}
    If we consider $(R, c) = (F[\sqrt{t}], 2\sqrt{t})$, the subring $R_0$ is $F[t]$ and $H_{0, (c/2)^2}(D; R_0) = H_{0, t}(D; F[t])$. From \Cref{prop:H(D)-basis}, $H_{0, t}(D; F[t])$ is freely generated by $\{ [\cz], X[\cz] \}$ over $F[t]$. The endomorphism $X$ gives $H_{0, t}(D; F[t])$ an $F[X]$-module structure. With $X^2 = (c/2)^2 = t$ we see that it is freely generated by $[\cz]$ over $F[X]$. From \Cref{prop:H(D)-basis} we have
    \[
        \qdeg_h{[\cz]} = \qdeg_h{[\ca_h]} + 2(k + 1) = \s_h(K; F[h]) + 1.
    \]
\end{proof}

The following one is given by Kronheimer and Mrowka in \cite[Section 2.2]{Kronheimer:2011by} based on the above definition of $s$.

\begin{corollary}
    Let $K$ be a knot. Take any connected cobordism $S$ from the unknot $U$ to $K$. Let $D$ be a diagram of $K$ and
    \[
        \phi: A \rightarrow H_h(D; F[h])_f
    \] 
    be the homomorphism obtained from $S$. With $1, X \in A$, define
    \begin{align*}
        m^+ &= k_h( \phi(1) ), \\
        m^- &= k_h( \phi(X) ).
    \end{align*}
    Then
    \[
        s(K; F) = m_+ + m_- + \chi(S).
    \]
\end{corollary}

\begin{proof}
    Since $[\cz(U)] = 1$ and $\s_h(U) = 0$, we either have
    \[
        m_+ = m_- = \frac{\s_h(K) - \chi(S)}{2},
    \]
    or
    \[
        m_+ + 1 = m_- - 1 = \frac{\s_h(K) - \chi(S)}{2}.
    \]
\end{proof}

We end this section with the following questions.

\begin{question}
    Can we extend \Cref{thm:s-with-k} to links?
\end{question}

\begin{question}\label{quest:s_c-eq-s}
    Does $\s_c$ coincide with $s$ for any $(R, c)$?
\end{question}

It is a famous open question whether there exists any $F$ such that $s(-; F)$ is distinct from $s = s(-; \QQ)$ (\cite[Question 6.1]{lipshitz2014refinement}). \Cref{mainthm3} implies that if \Cref{quest:s_c-eq-s} is solved affirmatively, then all $s(-; F)$ are equal among fields $F$ of $\fchar{F} \neq 2$.

\begin{remark} \label{rem:s-F2}
    In \cite{lipshitz2014refinement}, an alternative definition of the $s$-invariant for knots over a field $F$ (including $\fchar{F} = 2$) is given based on the filtered Bar-Natan homology:
    \[
        s'(K; F) = \frac{q_\mathit{min} + q_\mathit{max} }{2}
    \]
    where
    \begin{align*}
        q_\mathit{min} &= \min\{ \qdeg{x} \mid x \in H_{f\BN}(D; F) \setminus 0 \}, \\
        q_\mathit{max} &= \max\{ \qdeg{x} \mid x \in H_{f\BN}(D; F) \setminus 0 \}.
    \end{align*}
    \Cref{prop:ht-relation} implies that this definition coincides with \Cref{def:ras-s} when $\fchar{F} \neq 2$. For $F = \FF_2$, Seed showed by direct computation that $K = K14n19265$ has $s(K; \QQ) = 0$ but $s'(K; \FF_2) = -2$ (see \cite[Remark 6.1]{lipshitz2014refinement}).
\end{remark}

    \section{Further remarks and questions} \label{sec:further}

\bulsubsection{Implication from the Jones' conjecture}
    $k_c(D)$ can be related with the Jones' conjecture, a classical conjecture in knot theory which is now resolved affirmatively. It was proposed by Jones in \cite{Knots:fw}, reformulated by Male{\v s}i{\v c} and Traczyk in \cite{Malesic:wm} and by Kawamuro in \cite{kawamuro2006_1}, \cite{kawamuro2006_2}, proved by Dynnikov and Prasolov in \cite{Dynnikov:ui} and independently by LaFountain and Menasco in \cite{LaFountain:uf}.
    
    \begin{theorem}[Jones Conjecture \cite{Knots:fw}] \label{thm:jones-conj}
        If $D_0$ is a diagram of an oriented link $L$ having the minimum number of Seifert circles $r_0$ among all diagrams of $L$, then
        \begin{enumerate}
            \item $w(D_0)$ is uniquely determined.
            \item For any diagram $D$ of L with $r_0 + m$ Seifert circles, $w(D)$ is bounded as:
                $$ w(D_0) - m \leq w(D) \leq w(D_0) + m. $$
        \end{enumerate}
    \end{theorem}
    
    Combining this result with the definition of $\s_c$, we obtain:
    
    \begin{proposition}
        With the assumption of \Cref{thm:jones-conj}:
        \begin{enumerate}
            \item $k_c(D_0)$ is uniquely determined.
            \item For any diagram $D$ of L with $r_0 + m$ Seifert circles, $k_c(D)$ is bounded as:
                $$ k_c(D_0) \leq k_c(D) \leq k_c(D_0) + m. $$
        \end{enumerate}
    \end{proposition}
    
    Thus $k_c$ takes the minimum value whenever $r$ is minimum.

\bulsubsection{Transverse link invariants} 

    A \textit{transverse link} is a link in $\RR^3$ that is everywhere transverse to the standard contact structure $\ker(dz - y dx)$. From a work of Bennequin \cite{Bennequin:eK9ZWAXw}, given a braid representation $B$ of a transverse link $T$, the \textit{self-linking number} of $T$ is given by
    \[
        \mathit{sl}(T) = -b(B) + e(B)
    \]
    where $b(B)$ is the number of strings of $B$ and $e(B)$ is the exponent sum. Denoting by $D$ the closure of $B$, we have $b(B) = r(D)$ and $e(B) = w(D)$. From \Cref{prop:cacb-variance-under-rho}, after redefining $\rho$ with $\epsilon\rho$, we obtain
    \begin{proposition} \label{prop:lee-class-trans-inv}
        $[\ca(D)] \in H_c(D; R)$ is an invariant of $T$.
    \end{proposition}
    
    The special case with $c = 0$ (Khovanov homology) gives Plamenevskaya's invariant $\psi(T)$ (\cite{Plamenevskaya:2006dd}). We know from \Cref{prop:cacb-variance-under-rho} that the effect of a negative twist is multiplication by $c = 0$, so we see that a transverse stabilization annihilates $\psi(T)$ (\cite[Theorem 3]{Plamenevskaya:2006dd}). The special case $c = 1$ (filtered Bar-Natan homology) is given by Lipshitz in \cite{Lipshitz:2013kp}, and for the general case is given by Collari in \cite{Collari:2017wr}. 
    
    Numerical transverse link invariants can be extracted from $[\ca(D)]$. We already have $k_c(D)$, and also we can define
    \[
        \tilde{k}_c(D; R) = k_c([\ca(D)])\ \text{where } [\ca(D)] \in H_c(D; R).
    \]
    Note that $\tilde{k}_c$ measures the $c$-divisibility in $H_c(D; R)$, whereas $k_c$ measures in the free part $H_c(D; R)_f$. Both are non-negative transverse link invariants, and obviously $\tilde{k}_c \leq k_c$. The special case of $\tilde{k}_c(T)$ for $(R, c) = (F[h], h)$ is given by Collari in \cite{Collari:2017wr} where it is called the \textit{c-invariant} of a transverse link. 
    
    \medskip
    
    \begin{proposition}
        \[ 
            \s_c(T) + 2e^-(T) - 1
                \ \leq\ 
                \mathit{sl}(T) 
                \ \leq\ 
                \s_c(T) - 1.  
        \]
        where 
        \[
            e^-(T) = \max_B \{\ e^-(B)\ \} \ \leq \ 0
        \]
        is the maximum negative exponent sum among all braids representing $T$.
    \end{proposition}
    \begin{proof}
        Obvious from $\mathit{sl}(T) = \s_c(T) - 2k_c(T) - 1$ and \Cref{cor:k_c-bound}.
    \end{proof}

\bulsubsection{Quasi-positive links / knots}

    \Cref{prop:k_posD}, \ref{prop:s-pos-link} and \Cref{thm:s_c-properties} can be generalized to \textit{quasi-positive links}. A braid $B$ is \textit{quasi-positive} if it is of the form
    \[
        B = \prod_k \omega_k \sigma_{i_k} \omega_k^{-1}
    \]
    where each $\sigma_{i_k}$ is one of the positive generators and each $\omega_k$ is a word in the braid group. A diagram $D$ is \textit{quasi-positive} if it is the closure of a quasi-positive braid.
    
    \begin{lemma}
        If $D$ is quasi-positive, then $k_c(D) = 0$.
    \end{lemma}
    
    \begin{proof}
        Take $B$ as above, and let 
        \[
            B' = \prod_k \omega_k \omega_k^{-1}.
        \]
        Let $D'$ be closure of $B'$. From \Cref{prop:cross-rm-bound}, we have $0 \leq k_c(D) \leq k_c(D')$. Cancellations of $\omega_k \omega_k^{-1}$ preserve $k_c$, so $k_c(D') = k_c(\bigcirc^{\sqcup r}) = 0$.
    \end{proof}
    
    \begin{proposition}
        If $K$ is a quasi-positive knot, 
        \[
            \s_c(K) = 2g_*(K) = 2g(K).
        \]
        \qed
    \end{proposition}

\bulsubsection{The canonical generator $[\cz(K)]$}

Rasmussen called the $\ca$-classes in $\QQ$-Lee theory the ``canonical generators" of $H_\Lee(L; \QQ)$, from the fact that they form a basis of $H_\Lee(D; \QQ)$ and that they are invariant (up to unit) under the Reidemeister moves. For a general $(R, c)$ we have seen that this does not hold (\Cref{prop:ab-gen}, \ref{prop:cacb-variance-under-rho}). In \Cref{subsec:coincidence-with-s} we considered $(R, c) = (F[h], h)$, and constructed the class $[\cz]$ in the proof of  \Cref{prop:H(D)-basis}. We claim that is reasonable to call $[\cz]$ the canonical generator of $H_{0, (h/2)^2}(K; R)_f$. If we redefine the isomorphism $\rho$ of \Cref{prop:cacb-variance-under-rho} with $\epsilon \rho$, we see that the induced map
\[
    \rho: H(D)_f \rightarrow H(D')_f
\]
is an $A$-module isomorphism that maps $[\cz(D)]$ to $[\cz(D')]$. Thus by regarding $[\cz(D)]$ as an element of $H(K)_f$ we have

\begin{proposition} \label{thm:H_c(K)_f canon-generator}
    $H_{0, (h/2)^2}(K; F[h])_f$ is generated by $[\cz(K)]$ over $A$.
\end{proposition}

Suppose $S$ is a cobordism between two knots $K, K'$. The corresponding homomorphism
\[
    \phi: H(D) \rightarrow H(D')
\]
can similarly be adjusted so that 
\[
    \phi[\ca] = c^l [\ca'] 
    \quad \text{where} \quad
    l = \frac{1}{2}(-\Delta r + \Delta w - \chi(S)).
\]
With this modification, we can prove that $\phi$ maps $[\zeta(D)]$ to 
\[
    (2X)^{\frac{\Delta \s - \chi(S)}{2}}[\zeta(D')]
\]
where $\Delta \s = \s_h(K') - \s_h(K)$. Since the result depends only on $K, K'$ and $S$, we obtain a well-defined map
\[
    \phi: H(K)_f \rightarrow H(K')_f.
\]
This map is also natural with respect to cobordisms, since both $\Delta \s$ and $\chi(S)$ are additive under compositions of cobordisms. Thus

\begin{proposition} \label{thm:cz-cobordism}
    $H_{0, (h/2)^2}(K; F[h])_f$ is a functor from the category of knots (with morphisms cobordisms between knots) to the category of $A$-modules.
\end{proposition}

In particular if $S$ is an annulus, then $[\zeta(K)]$ is mapped to $[\zeta(K')]$, so

\begin{proposition} \label{thm:cz-knot-concordance-inv}
    $[\zeta(K)]$ is a knot concordance invariant.
\end{proposition}

By collapsing $h \mapsto 2$, we obtain the corresponding propositions for Lee theory over $F$. Also with $H_{0, (h/2)^2}(-; F[h]) \isom H_\BN(-; F)$, we obtain those for Bar-Natan theory over $F$. Same arguments also hold for $(R, c) = (\ZZ, 2)$, the integral Lee theory. So in these cases, we have the canonical generator (as an $A$-module) that is strictly invariant under Reidemeister moves, and also invariant under knot concordance.

\begin{question} 
    Can we find such a class for a general $(R, c)$?
\end{question}

\begin{question} 
    Is there any geometric explanation for the class $[\cz(K)]$? 
\end{question}

    \appendix
\section{Proof of \Cref{prop:cacb-variance-under-rho}}
\label{sec:canon-classes-and-RMs}

Let $R$ be a commutative ring with $h, t \in R$ and $c = \sqrt{h^2 + 4t}$. Let $D, D'$ be diagrams related by a Reidemeister move. There is a quasi-isomorphism corresponding to the move
\[
    \rho: C_{h, t}(D; R) \rightarrow C_{h, t}(D'; R).
\]
Denote by $\ca, \ca'$ the $\ca$-classes of $D, D'$, and by $\cb, \cb'$ the $\cb$-classes of $D, D'$ respectively. We prove
\[
    \ca' \sim \epsilon c^j \rho(\ca), \quad
	\cb' \sim \epsilon' c^j \rho(\cb)
\]
where $j \in \{ \pm 1 \}$ is given as in \Cref{table:RM-k-corresp}, and $\epsilon \epsilon' = (-1)^j$. We define $\rho$ by modifying the isomorphisms given in \cite[Section 5]{khovanov2000}. We omit the subscript $(h, t)$ and the ring $R$ in the remaining.


\bulsubsubsection{RM1$_L$ : Left twist}

\begin{figure}[ht]
    \centering
    \includegraphics[scale=0.3]{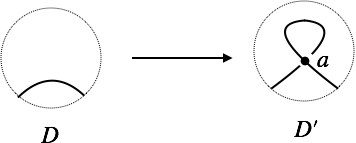}
	\caption{A left twist} \label{fig:RM1}
\end{figure}

Let $D'$ be the diagram obtained by performing a left-twist on an arc of $D$ (\Cref{fig:RM1}). Let $a$ be the added crossing of $D'$, and $D'_0, D'_1$ be the 0-, 1- resolved diagram of $D'$ at $a$ respectively . There is a decomposition of $\bar{C}(D')$ as in \Cref{fig:RM1_C}. 

\begin{figure}[ht]
	\centering
    \includegraphics[scale=0.3]{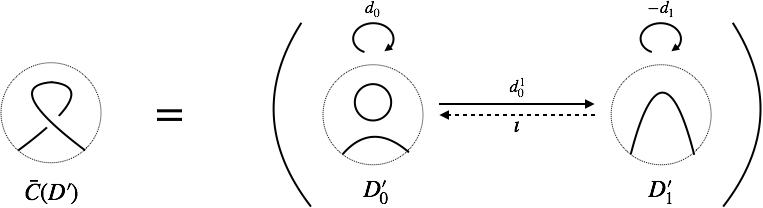}
	\caption{Decomposition of $\bar{C}(D')$} \label{fig:RM1_C}
\end{figure}

Define a bidegree $(-1, 0)$ chain map
\begin{align*}
	\iota : \bar{C}(D'_1)[1, 1] 
	    &\longrightarrow 
	    \bar{C}(D'_0) \isom \bar{C}(D'_1) \otimes A \\
    x
        &\longmapsto
        x \otimes 1.
\end{align*}

\noindent
and a bidegree preserving chain map
\[
	\gamma = \iota \circ d^1_0 : \bar{C}(D'_0) \rightarrow \bar{C}(D'_0).
\]

\noindent
Also define a subset of $\bar{C}(D')$ by
\[
    X_1 = \{ x - \gamma(x) \mid x \in \bar{C}(D'_0) \}.
\]
It can be shown that $X_1$ is a subcomplex of $\bar{C}(D')$ and that $\bar{C}(D')$ decomposes into the direct sum of $X_1$ and an acyclic subcomplex. There is a bigrading preserving isomorphism from $C(D)$ to $X_1$ given by
\begin{align*}
    \rho: C(D) &\longrightarrow X_1 \\
        X &\longmapsto X \otimes X - hX \otimes 1 - t1 \otimes 1 \\
        1 &\longmapsto 1 \otimes X - X \otimes 1.
\end{align*}

This induces a bigrading preserving isomorphism
\[
    \rho: H(D) \rightarrow H(D').
\]

\begin{figure}[ht]
    \centering
    \includegraphics[scale=0.35]{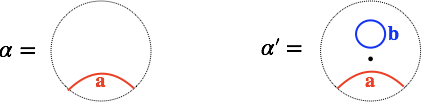}
    \caption{Colorings of $\ca$ and $\ca'$}
\end{figure}

Now, the Seifert circle of $D$ that intersects the interior of \dottedcircle is either colored $\a$ or $\b$. We may assume the first, since the other case is considered by $\cb$. Let $\ca = \cdots \otimes \a$ and $\ca' = \cdots \otimes \a \otimes \b$. Recall that $\a = X - u,\ \b = X - v$ and $h = u + v,\ t = -uv$. By direct calculation, $\rho$ maps
\begin{align*}
    \a &\mapsto \a \otimes \b \\
    \b &\mapsto \b \otimes \a,
\end{align*}
so $\ca' = \rho(\ca),\ \cb' = \rho(\cb)$.


\bulsubsubsection{RM1$_R$ : Right twist}

A right twist is accomplished by a composition of tangency move (RM2) and a left untwist (RM1$_L^{-1}$), so the result follows from those of the two moves.


\bulsubsubsection{RM2 : Tangency move} 

\begin{figure}[ht]
    \centering
    \includegraphics[scale=0.35]{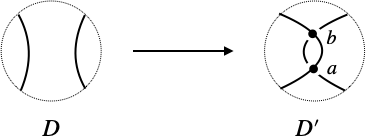}
	\caption{A tangency move} \label{fig:RM2}
\end{figure}

Let $D$, $D'$ be two diagrams as depicted in \Cref{fig:RM2}. We take a crossing-order of $D'$ so that $a, b$ are placed in this order at the end. There is a decomposition of $\bar{C}(D')$ as in \Cref{fig:RM2_C}. 
\begin{figure}[ht]
    \centering
    \includegraphics[scale=0.3]{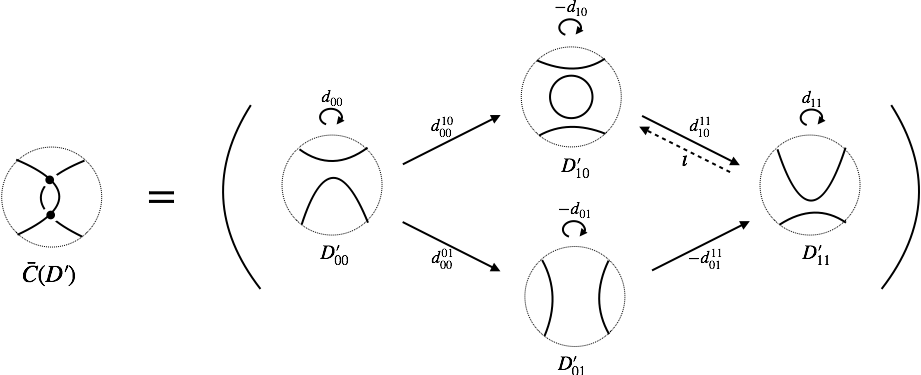}
	\caption{Decomposition of $\bar{C}(D')$} \label{fig:RM2_C}
\end{figure}
Define chain maps as in the previous case: 
\begin{align*}
	\iota : \bar{C}(D'_{11})[2, 2] &\longrightarrow \bar{C}(D'_{10})[1, 1] \\
	\gamma = \iota \circ d_{01}^{11} : \bar{C}(D'_{01})[1, 1] &\longrightarrow \bar{C}(D'_{10})[1, 1]
\end{align*}

Also define
\[
    X_1 = \{ x + \gamma(x) \mid x \in \bar{C}(D'_{01})[1, 1] \}.
\]

As in the previous case, we have $X_1 \htpy C(D')$ and the quasi-isomorphism $\rho$ is given by
\begin{align*}
    \rho : C(D) &\longrightarrow X_1 \\
              x &\longmapsto     x + \gamma(x)
\end{align*}

We divide cases by the direction of the two arcs of $D$ in \dottedcircle. Let $s, s'$ the orientation preserving states of $D, D'$ respectively.

\begin{case} [The two arcs points to the same direction]
	$\Delta r = 0$. Since $C(D) = C(D'_{01})$ and $\gamma(\ca) = \gamma(\cb) = 0$, we have $\ca' = \rho(\ca),\ \cb' = \rho(\cb)$.
\end{case}

\begin{case}[The two arcs points to the opposite direction]
	We split cases depending on whether the two arcs of $\ca$ seen in \dottedcircle belongs to the same $s$-circle or not.
    
    \begin{subcase} [The two arcs belong to the same $s$-circle]
		$\Delta r = 2$. 
	    \begin{figure}[H]
            \centering            \includegraphics[scale=0.35]{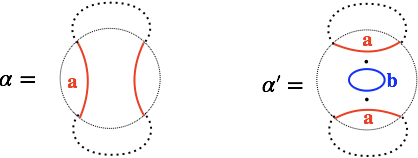}
        \end{figure}
		Take an element $x$ in $C(D'_{00})$ as in \Cref{fig:RM2_x}. 
        \begin{figure}[ht]
            \centering
            \includegraphics[scale=0.3]{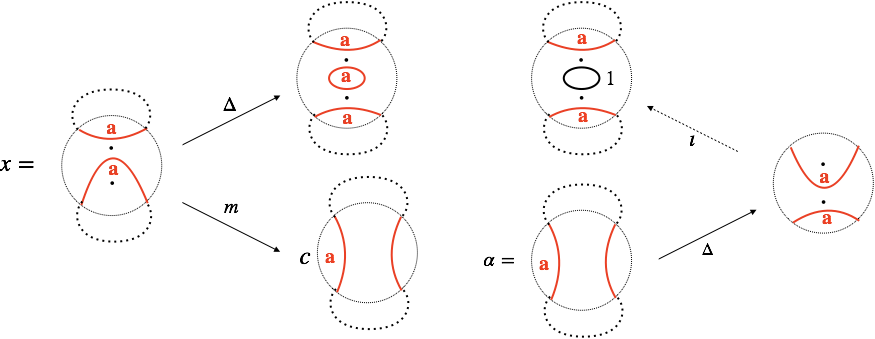}
            \caption{The element $x$ in $C(D)$}
            \label{fig:RM2_x}
        \end{figure}
		From $\a = \b + c$, we have
        \[
            dx = (\ca' + c\gamma(\ca)) + c\ca \ 
                \Rightarrow \ 
                \ca' \sim -c\rho(\ca).
        \]
        
        Let $\bar{x}$ be the chain obtained from $x$ by flipping $\a$'s and $\b$'s. Then
        \begin{align*}
            d\bar{x} = (\cb' - c \gamma(\cb)) - c \cb \ 
                &\Rightarrow \ 
                \cb' \sim c\rho(\cb)
        \end{align*}
    \end{subcase}

    \begin{subcase}[The two arcs belong to the different $s$-circles] 
        Similarly we obtain $\Delta r = 0$ and
        \[
            \ca' \sim -\rho(\ca), 
            \quad
            \cb' \sim -\rho(\cb)
        \]
    \end{subcase}
\end{case}


\bulsubsubsection{RM3 : Triple point move}

\begin{figure}[H]
    \centering
    \includegraphics[scale=0.35]{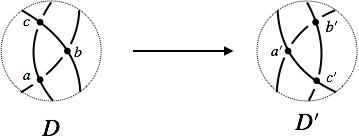}
	\caption{A triple point move} \label{fig:RM3}
\end{figure}

Let $D$, $D'$ be two diagrams as depicted in \Cref{fig:RM3}. Fix any crossing-order of $D$ so that $a, b, c$ are placed in this order at the end (and similarly for $D'$). The three crossings are taken so that $D_{**1}$ and $D'_{**1}$ are isotopic. 

\begin{figure}[ht]
    \centering
    \includegraphics[scale=0.35]{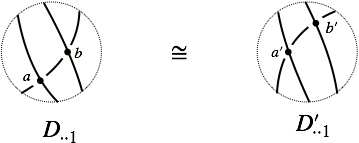}
\end{figure}

\begin{figure}[t]
    \centering
    \includegraphics[scale=0.35]{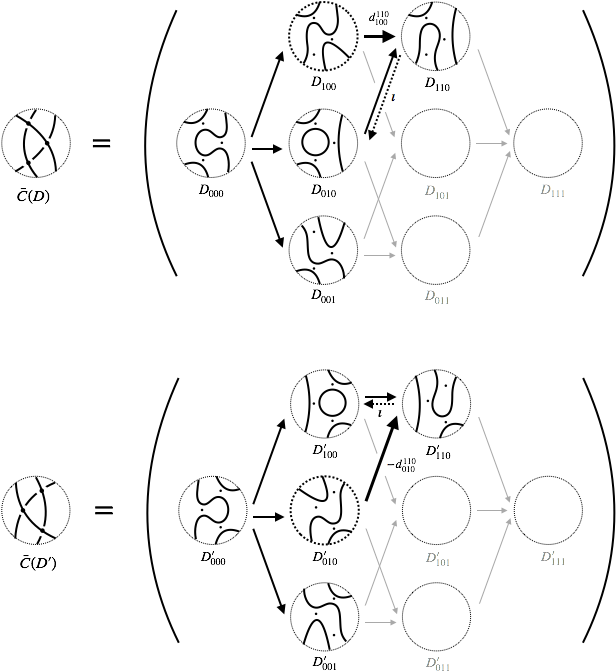}
	\caption{Decompositions of $\bar{C}(D)$ and $\bar{C}(D')$} \label{fig:RM3_C}
\end{figure}

There are decompositions of $\bar{C}(D), \bar{C}(D')$ as in \Cref{fig:RM3_C}. Define maps
\[
    \left\{ \begin{array}{ccc}
    	\iota : \bar{C}(D_{110})[2, 2] 
    	    &\longrightarrow 
    	    &\bar{C}(D_{010})[1, 1] \\
    	\gamma : \bar{C}(D_{100})[1, 1] 
    	    &\longrightarrow 
    	    &\bar{C}(D_{010})[1, 1]
    \end{array} \right.
\]

\[
    \left\{ \begin{array}{ccc}
    	\iota': \bar{C}(D'_{110})[2, 2] 
    	    &\longrightarrow 
    	    &\bar{C}(D_{100})[1, 1] \\
    	\gamma': \bar{C}(D'_{010})[1, 1] 
    	    &\longrightarrow 
    	    &\bar{C}(D_{100})[1, 1] \\
    \end{array} \right.
\]

\noindent
and subsets of $\bar{C}(D), \bar{C}(D')$ by
\begin{align*}
    X_1 &= \{ x + \gamma(x) + y \mid x \in \bar{C}(D_{100})[1, 1], y \in \bar{C}(D_{**1}) \} \\
    X'_1 &= \{ x + \gamma'(x) + y \mid x \in \bar{C}(D'_{010})[1, 1], y \in \bar{C}(D'_{**1}) \}. 
\end{align*} 

Again we have $X_1 \htpy C(D),\ X'_1 \htpy C(D')$, and the isomorphism $\rho$ is given by:
\begin{align*}
	\rho : X_1 &\longrightarrow X'_1 \\
           x + \gamma(x) + y &\longmapsto x + \gamma'(x) + y
\end{align*}

Since the move is point-symmetric in \dottedcircle, regarding directions of the strands, we may assume that the top-most strand of $D$ directs upward. Thus there are four possible cases for the directions of the other two strands. For the following three cases: $\uparrow\uparrow\uparrow$, $\uparrow\uparrow\downarrow$ and $\uparrow\downarrow\downarrow$, we see that $\Delta r = 0$, and from the definition of $\rho$ we have $\ca' = \rho(\ca), \cb' = \rho(\cb)$.
The remaining case is  $\uparrow\downarrow\uparrow$. There are five possible subcases regarding the connections of the arcs:

\begin{figure}[H]
    \centering
    \includegraphics[scale=0.35]{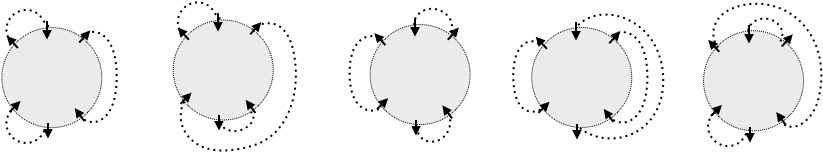}
\end{figure}

\setcounter{case}{0}
\begin{case}
    $\Delta r = -2$. Suppose $\ca, \ca'$ are colored as follows:
    \begin{figure}[H]
        \centering
        \includegraphics[scale=0.35]{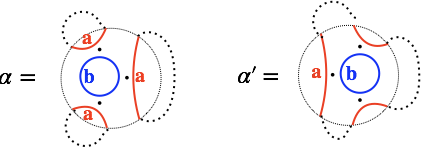}
    \end{figure}
    Define elements $x, y, z \in C(D)$ as in \Cref{fig:RM3_xyz}. 
    \begin{figure}[ht]
        \centering
        \includegraphics[scale=0.3]{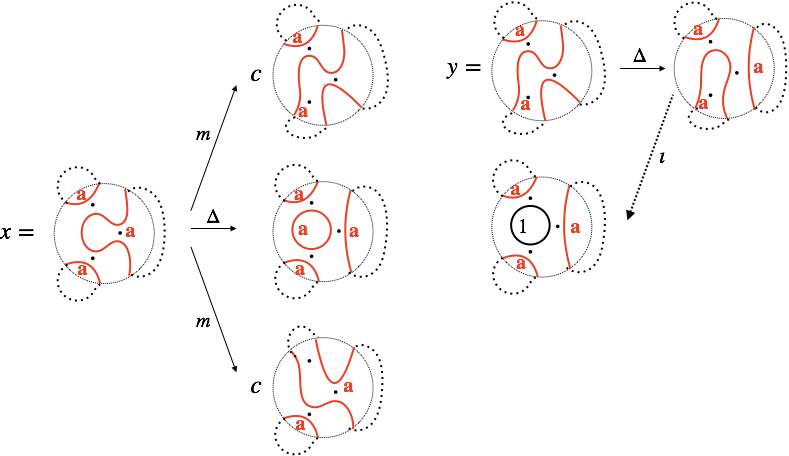}
        \caption{Elements $x, y, z \in C(D)$}
        \label{fig:RM3_xyz}
    \end{figure}

    Then
    \begin{align*}
        dx = cy + (\ca + c\gamma(y)) + cz 
            &\Rightarrow\ 
            \ca \sim -cy - c\gamma(y) - cz \\
        d\bar{x} = -c\bar{y} + (\cb - c\gamma(\bar{y})) - c\bar{z} 
            &\Rightarrow\ 
            \cb \sim c\bar{y} + c\gamma(y) + cz
    \end{align*}

	Similarly in $C(D')$, define chains $x', y', z'$ as in \Cref{fig:RM3_xyz_2}. 
	\begin{figure}[ht]
        \centering
        \includegraphics[scale=0.3]{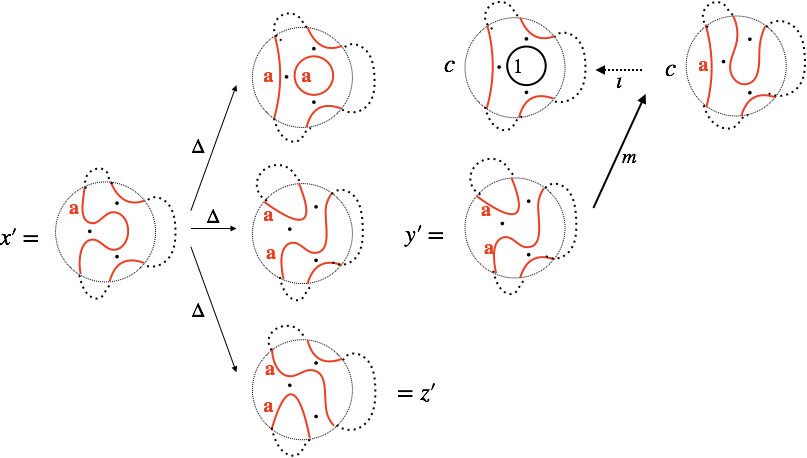}
        \caption{Elements $x', y', z' \in C(D')$}
        \label{fig:RM3_xyz_2}
    \end{figure}

    Then
    \begin{align*}
        dx' = (\ca' + \gamma(y')) + y' + z' &\Rightarrow\ \ca' \sim -y' - \gamma(y') - z' \\
        d\bar{x}' = (\cb' + \gamma(\bar{y}')) + \bar{y}' + \bar{z}' &\Rightarrow\ \cb' \sim -\bar{y}' - \gamma(\bar{y}') - \bar{z}'
    \end{align*}

    Thus from the definition of $\rho$, we have
    \[
        c \ca' \sim \rho(\ca),
        \quad
        -c \cb' \sim \rho(\cb).
    \]
\end{case}

The remaining four cases proceeds similarly. \qed

    \newpage
    \bibliographystyle{plain}
    \bibliography{bibliography.bib}

\begin{thebibliography}{10}

\bibitem{Alishahi:2017ug}
Akram {Alishahi}.
\newblock {The Bar-Natan homology and unknotting number}.
\newblock {\em arXiv e-prints}, page arXiv:1710.07874, October 2017.

\bibitem{Alishahi:2017wl}
Akram Alishahi and Nathan Dowlin.
\newblock The lee spectral sequence, unknotting number, and the knight move
  conjecture.
\newblock {\em Topology and its Applications}, 2018.

\bibitem{BarNatan:2004if}
Dror Bar-Natan.
\newblock Khovanov's homology for tangles and cobordisms.
\newblock {\em Geom. Topol.}, 9:1443--1499, 2005.

\bibitem{beliakova2008categorification}
Anna Beliakova and Stephan Wehrli.
\newblock Categorification of the colored {J}ones polynomial and {R}asmussen
  invariant of links.
\newblock {\em Canad. J. Math.}, 60(6):1240--1266, 2008.

\bibitem{Bennequin:eK9ZWAXw}
Daniel Bennequin.
\newblock Entrelacements et \'{e}quations de {P}faff.
\newblock In {\em Third {S}chnepfenried geometry conference, {V}ol. 1
  ({S}chnepfenried, 1982)}, volume 107 of {\em Ast\'{e}risque}, pages 87--161.
  Soc. Math. France, Paris, 1983.

\bibitem{Collari:2017bennequin}
Carlo {Collari}.
\newblock {A Bennequin-type inequality and combinatorial bounds}.
\newblock {\em arXiv e-prints}, page arXiv:1707.03424, Jul 2017.

\bibitem{Collari:2017wr}
Carlo Collari.
\newblock Transverse invariants from {K}hovanov-type homologies.
\newblock {\em J. Knot Theory Ramifications}, 28(1):1950012, 37, 2019.

\bibitem{Dynnikov:ui}
I.~A. Dynnikov and M.~V. Prasolov.
\newblock Bypasses for rectangular diagrams. {A} proof of the {J}ones
  conjecture and related questions.
\newblock {\em Trans. Moscow Math. Soc.}, pages 97--144, 2013.

\bibitem{Hedden:2012iz}
Matthew Hedden and Yi~Ni.
\newblock Khovanov module and the detection of unlinks.
\newblock {\em Geom. Topol.}, 17(5):3027--3076, 2013.

\bibitem{Knots:fw}
V.~F.~R. Jones.
\newblock Hecke algebra representations of braid groups and link polynomials.
\newblock {\em Ann. of Math. (2)}, 126(2):335--388, 1987.

\bibitem{kawamuro2006_1}
Keiko Kawamuro.
\newblock The algebraic crossing number and the braid index of knots and links.
\newblock {\em Algebr. Geom. Topol.}, 6:2313--2350, 2006.

\bibitem{kawamuro2006_2}
Keiko Kawamuro.
\newblock Conjectures on the braid index and the algebraic crossing number.
\newblock In {\em Intelligence of low dimensional topology 2006}, volume~40 of
  {\em Ser. Knots Everything}, pages 151--155. World Sci. Publ., Hackensack,
  NJ, 2007.

\bibitem{khovanov2000}
Mikhail Khovanov.
\newblock A categorification of the {J}ones polynomial.
\newblock {\em Duke Math. J.}, 101(3):359--426, 2000.

\bibitem{Khovanov:2002wo}
Mikhail Khovanov.
\newblock Patterns in knot cohomology. {I}.
\newblock {\em Experiment. Math.}, 12(3):365--374, 2003.

\bibitem{khovanov2004}
Mikhail Khovanov.
\newblock Link homology and {F}robenius extensions.
\newblock {\em Fund. Math.}, 190:179--190, 2006.

\bibitem{Kronheimer:1993}
P.~B. Kronheimer and T.~S. Mrowka.
\newblock Gauge theory for embedded surfaces. {I}.
\newblock {\em Topology}, 32(4):773--826, 1993.

\bibitem{Kronheimer:2011by}
P.~B. Kronheimer and T.~S. Mrowka.
\newblock Gauge theory and {R}asmussen's invariant.
\newblock {\em J. Topol.}, 6(3):659--674, 2013.

\bibitem{LaFountain:uf}
Douglas~J. LaFountain and William~W. Menasco.
\newblock Embedded annuli and {J}ones' conjecture.
\newblock {\em Algebr. Geom. Topol.}, 14(6):3589--3601, 2014.

\bibitem{lee2005endomorphism}
Eun~Soo Lee.
\newblock An endomorphism of the {K}hovanov invariant.
\newblock {\em Adv. Math.}, 197(2):554--586, 2005.

\bibitem{lewark2009rasmussen}
Lukas Lewark.
\newblock {\em The Rasmussen invariant of arborescent and of mutant links}.
\newblock PhD thesis, Master thesis, ETH Z{\"u}rich, 2009.

\bibitem{Lipshitz:2013kp}
Robert Lipshitz, Lenhard Ng, and Sucharit Sarkar.
\newblock On transverse invariants from {K}hovanov homology.
\newblock {\em Quantum Topol.}, 6(3):475--513, 2015.

\bibitem{lipshitz2014refinement}
Robert Lipshitz and Sucharit Sarkar.
\newblock A refinement of {R}asmussen's {$S$}-invariant.
\newblock {\em Duke Math. J.}, 163(5):923--952, 2014.

\bibitem{mackaay2007remark}
Marco Mackaay, Paul Turner, and Pedro Vaz.
\newblock A remark on {R}asmussen's invariant of knots.
\newblock {\em J. Knot Theory Ramifications}, 16(3):333--344, 2007.

\bibitem{Malesic:wm}
Jo\v{z}e Male\v{s}i\v{c} and Pawe\l Traczyk.
\newblock Seifert circles, braid index and the algebraic crossing number.
\newblock {\em Topology Appl.}, 153(2-3):303--317, 2005.

\bibitem{Milnor1968}
John Milnor.
\newblock {\em Singular points of complex hypersurfaces}.
\newblock Annals of Mathematics Studies, No. 61. Princeton University Press,
  Princeton, N.J.; University of Tokyo Press, Tokyo, 1968.

\bibitem{Mukherjee:ww}
Sujoy Mukherjee, Józef~H. Przytycki, Marithania Silvero, Xiao Wang, and
  Seung~Yeop Yang.
\newblock Search for torsion in khovanov homology.
\newblock {\em Experimental Mathematics}, 0(0):1--10, 2017.

\bibitem{Plamenevskaya:2006dd}
Olga Plamenevskaya.
\newblock Transverse knots and {K}hovanov homology.
\newblock {\em Math. Res. Lett.}, 13(4):571--586, 2006.

\bibitem{Rasmussen:2005vo}
Jacob {Rasmussen}.
\newblock {Khovanov's invariant for closed surfaces}.
\newblock {\em arXiv Mathematics e-prints}, page math/0502527, February 2005.

\bibitem{rasmussen2010khovanov}
Jacob Rasmussen.
\newblock Khovanov homology and the slice genus.
\newblock {\em Invent. Math.}, 182(2):419--447, 2010.

\bibitem{Shumakovitch:2007}
Alexander~N. Shumakovitch.
\newblock Rasmussen invariant, slice-{B}ennequin inequality, and sliceness of
  knots.
\newblock {\em J. Knot Theory Ramifications}, 16(10):1403--1412, 2007.

\bibitem{Wehrli:uy}
S.~Wehrli.
\newblock A spanning tree model for {K}hovanov homology.
\newblock {\em J. Knot Theory Ramifications}, 17(12):1561--1574, 2008.

\end{thebibliography}

\end{document}